\documentclass[reqno,12pt,letterpaper]{amsart}

\usepackage{amsmath,amssymb,amsthm,graphicx,mathrsfs,url,bm,subfigure,accents,paralist,xargs,slashed,enumitem,mathtools,array}
\usepackage[usenames,dvipsnames]{xcolor}
\usepackage[colorlinks=true,linkcolor=Red,citecolor=Green]{hyperref}
\usepackage[colorinlistoftodos,prependcaption,textsize=tiny]{todonotes}
\presetkeys{todonotes}{inline}{}

\usepackage{tikz}
\usepackage{pgfplots}
\usetikzlibrary{positioning}
\usetikzlibrary{decorations.markings}
\usetikzlibrary{calc}
\usetikzlibrary{patterns}
\usetikzlibrary{intersections}
\usetikzlibrary{backgrounds}
\usetikzlibrary{arrows}
\usepackage{tikz}
\usepackage{pgfplots}
\usetikzlibrary{positioning}
\usetikzlibrary{decorations.markings}
\usetikzlibrary{calc}
\usetikzlibrary{patterns}
\usetikzlibrary{intersections}
\usetikzlibrary{backgrounds}
\usetikzlibrary{arrows}
\numberwithin{equation}{section}

\newtheorem{theo}{Theorem}
\newtheorem{prop}{Proposition}[section]
\newtheorem{lemm}[prop]{Lemma}

\theoremstyle{definition}
\newtheorem{defi}[prop]{Definition}
\newtheorem{rema}[prop]{Remark}

\newcommand{\sol}{\mathfrak{B}}

\newcommand{\RR}{\mathbb{R}}
\newcommand{\CC}{\mathbb{C}}

\newcommand{\GD}{\mathrm{GD}}
\newcommand{\TM}{\mathrm{TM}}
\newcommand{\AG}{\mathrm{AG}}

\newcommand{\feas}{\mathrm{f}}

\newcommand{\C}{\mathcal{C}}

\newcommand{\sig}{\ell}
\newcommand{\IQC}{\mathrm{IQC}}

\newcommand{\conj}{\bar}

\newcommand{\refe}{\mathrm{ref}}
\newcommand{\disc}{\mathrm{disc}}

\renewcommand{\Re}{\operatorname{Re}}

\title[A frequency-domain analysis of inexact gradient methods]
{A frequency-domain analysis of inexact gradient methods}

\author{Oran Gannot}
\email{ogannot@berkeley.edu}

\begin{document} 
	
	\begin{abstract} 
We study robustness properties of some iterative gradient-based methods for strongly convex functions, as well as for the larger class of functions with sector-bounded gradients, under a relative error model. Proofs of the corresponding convergence rates are based on frequency-domain criteria for the stability of nonlinear systems. Applications are given to inexact versions of gradient descent and the Triple Momentum Method.  To further emphasize the usefulness of frequency-domain methods, we derive improved analytic bounds for the convergence rate of Nesterov's accelerated method (in the exact setting) on strongly convex functions.
	\end{abstract}
	
	\maketitle
	
	\thispagestyle{empty}
	
	\section{Introduction}

\subsection{Overview} \label{subsect:overview}
As observed in  \cite{lessard2016analysis}, control-theoretic techniques originating with the absolute stability problem provide a useful framework for the design and analysis of iterative first-order optimization methods. Certain fundamental algorithms, including gradient descent, can be viewed as a linear system in feedback with the gradient $\nabla f$ of the function to be optimized. More precisely, their iterates are of the form
\begin{equation} \label{eq:lure}
x_{t+1} = Ax_t + B\nabla f(Cx_t) 
\end{equation}
for matrices $A,B,C$. The purpose of this paper is to show how frequency-domain methods can be used to characterize robustness properties of such algorithms under a relative error model. 

 
To formulate the error model, suppose that instead of the exact gradient in \eqref{eq:lure} one measures $\nabla f(Cx_t) + e_t$, where the magnitude of the error $e_t$ is bounded by a multiple $\delta \in [0,1)$ of the exact gradient:
\begin{equation} \label{eq:errormodel}
|e_t| \leq \delta |\nabla f(Cx_t)|, \quad t \geq 0.
\end{equation}
If $f$ belongs to a class of functions for which the exact iterates converge exponentially (i.e., geometrically) to a unique equilibrium, then it is still possible for the inexact iterates to converge exponentially as well. Naturally, any convergence rate for the inexact problem is expected to degrade with the size of $\delta$. Since the error sequence $e$ can be chosen adversarially, upper bounds on the inexact rates provide a strong measure of the robustness of a given algorithm to gradient perturbations.  

 Following a long tradition in control theory, exponential stability at a prescribed rate for systems of the form \eqref{eq:lure} (often called Lur'e systems \cite{lurie1957some}) can be certified in the time domain via the feasibility of a certain linear matrix inequality (LMI) with two sets of decision variables:
\begin{enumerate} \itemsep6pt 
	\item a \emph{storage function} (a Lyapunov-type function), which is a quadratic form in the state variables,
	\item  multipliers corresponding to constraints satisfied by the nonlinearity (so-called \emph{integral quadratic constraints}, or IQCs).
\end{enumerate}
 For a fixed set of IQCs, we refer to this feasibility problem as a \emph{stability test} or \emph{stability criterion}. The frequency methods in this paper provide a systematic way of reducing the number of free parameters for a given stability test: the feasibility problem can alternatively be formulated in the frequency domain, wherein the storage function does not appear explicitly. This eliminates a number of decision variables that grows quadratically with the dimension of the state space. Furthermore, at the margins of stability (e.g., when searching for the best possible exponential convergence rate), the values of the multipliers are often determined uniquely by the frequency test.

These frequency methods are applied to inexact gradient descent, which exhibits richer phenomenology than its exact counterpart. We study convergence rates for inexact gradient descent over the set of strongly convex functions, as well as the larger class of functions with sector-bounded gradients. For a discussion of these function classes see \S \ref{subsect:main} below, but we stress that functions with sector-bounded gradients need not be convex. In the exact setting it is known that the worst-case exponential convergence rates over these two classes coincide. 

For inexact gradient descent our results are partially contained in  \cite{de2017worst1}, although the methodology differs. Using essentially the same stability tests but in the time-domain, the authors of \cite{de2017worst1} numerically solve the feasibility problem for a range of problem instantiations, and then guess analytic formulas for the decision variables. This work proposes a principled and potentially simpler approach for deriving these analytical results that does not depend on access to numerical solvers, nor on explicit constructions of Lyapunov functions.

At the opposite end of the complexity scale, we consider an inexact version of the Triple Momentum Method \cite{van2017fastest}. Although this is the fastest known first-order method for optimizing strongly convex functions when an exact oracle is available, numerical evidence suggests that it is not robust to perturbations \cite{cyrus2018robust}. We establish an explicit convergence rate for the inexact problem. Based on numerical evidence we conjecture that this rate is optimal with respect to the Jury--Lee criterion (the least conservative stability test for strongly convex functions considered in this paper).

As an introduction to frequency-domain analysis, we also revisit Nesterov's accelerated gradient method in the exact setting for strongly convex functions. The standard tuning and basic convergence rate for this method are given in \cite[\S 2.2]{nesterov2018lectures}. However, it was observed numerically in \cite{lessard2016analysis} that this basic rate is conservative. In \S \ref{subsect:nesterov}, we show how the Jury--Lee criterion can be used to easily improve upon the known rate by a constant factor.

\subsection{Results for inexact gradient descent} \label{subsect:main}
First we introduce the function classes used throughout the paper, which depend on two parameters $0 \leq m < L$.

\begin{defi} \label{defi:Sml}
	The gradient of a $\C^1$ function is said to be bounded in the sector $[m,L]$ if there exists $y_\star \in \RR^n$ such that 
		\begin{equation} \label{eq:sector}
	\langle m(y-y_\star) - \nabla f(y), L(y-y_\star) -\nabla f(y) \rangle \leq 0
	\end{equation}
for all $y \in \RR^n$.
	Denote by $\mathcal{S}(m,L)$ the set of all such functions. Let $\mathcal{S}_0(m,L)$ denote the subset of functions such that $y_\star = 0$ satisfies \eqref{eq:sector}.
\end{defi}
If $f \in \mathcal{S}(m,L)$, then $\nabla f(y_\star) = 0$ by continuity. Furthermore, if $m > 0$, then the reference point $y_\star$ is the unique global minimizer of $f$. Observe that functions in $\mathcal{S}(m,L)$ can be far from convex. This is easily seen in one dimension: if $f \in \mathcal{S}(m,L)$ and $y_\star = 0$, then the sector condition is equivalent to
\[
\min(my, Ly) \leq f'(y) \leq \max(my, Ly), 
\]
which places no slope restrictions on $f'$ at all. Sector-boundedness is also related to different notions of one-point convexity that arise in the literature; for some applications in neural networks, matrix completion, and phase retrieval, see \cite{kleinberg2018alternative,li2017convergence,arora2015simple,sun2016guaranteed,chen2015solving,candes2015phase,xiong2020analytical}. In the context of robust control, the question of stability for systems in feedback with sector-bounded nonlinearities is as old the subject itself \cite{lurie1957some}.

\begin{defi}
	Denote by $\mathcal{F}(m,L)$  the set  of $m$-strongly convex functions with $L$-Lipschitz gradients. Let $\mathcal{F}_0(m,L)$ denote the subset of functions such that $\nabla f(0) = 0$.
\end{defi}

Precisely, a $\C^1$ function $f : \RR^n \rightarrow \RR$ belongs to $\mathcal{F}(m,L)$ if  $x \mapsto f(x) - (m/2)\| x\|^2$ is convex and
\[
\| \nabla f(y_1) - \nabla f(y_2) \| \leq L\|y_1 -y_2\|
\]
for all $y_1, y_2 \in \RR^n$.
  It is well-known that membership in $\mathcal{F}(m,L)$ is equivalent to the strong monotonicity (or cocoercivity)  of its gradient \cite[Theorem 2.1.5]{nesterov2018lectures}. In other words, $f \in \mathcal F(m,L)$ if and only if
 	\begin{equation} \label{eq:monotone}
 \langle \nabla f (y_1)-\nabla f (y_2) - m(y_1 - y_2), \nabla f (y_1) - \nabla f(y_2) - L(y_1 - y_2) \rangle \leq 0 
 \end{equation}
for all $y_1, y_2 \in \RR^n$.
 In particular, $\mathcal{F}(m,L) \subset \mathcal{S}(m,L)$ and $\mathcal{F}_0(m,L) \subset \mathcal{S}_0(m,L)$.

Next, we recall  standard convergence properties of gradient descent on $\RR^n$ over the class $\mathcal{S}(m,L)$, where  $0 < m <L$. The iterates are given by 
\begin{equation} \label{eq:gradient}
x_{t+1} = x_t - \alpha \nabla f(x_t),
\end{equation}
where $\alpha > 0$. This can be written in the form \eqref{eq:lure} with $A = C = I_n$ and $B = -\alpha I_n$. Set 
\[
\rho_\GD = \rho_\GD(m,L,\alpha) \coloneqq \max(1-\alpha m, \alpha L -1).
\]
If $\alpha < 2/L$, then $\rho_\GD \in (0,1)$.
For a given $f \in \mathcal{S}(m,L)$, the gradient descent iterates satisfy the exponential stability estimate
\begin{equation} \label{eq:gradientlinearconvergence}
\| x_t -x_\star\| \leq \rho^t \|x_0 - x_\star\|, \quad t \geq 0
\end{equation}
with $\rho = \rho_\GD$. For a frequency interpretation of this bound see \S \ref{subsect:GD}. The optimal step size $\alpha = 2/(L+m)$ corresponds to the solution of $1 - \alpha m = \alpha L -1$, yielding the well-known rate
\[
\rho_\GD(m, L, 2/(L+m)) = \frac{\kappa -1}{\kappa + 1}.
\] 
Here $\kappa = L/m$ is the condition number. When $\alpha = 1/L$, one recovers the rate $\rho_\GD(m, L, 1/L) = 1-1/\kappa$.

If $f \in \mathcal{F}(m,L) \cap \C^2$, then the fact that \eqref{eq:gradientlinearconvergence}  holds with $\rho = \rho_\GD$ is classical \cite{polyak1987introduction}. As far as we are aware, the fact that this is also true for  $f \in \mathcal{S}(m,L)$ was first emphasized in \cite{lessard2016analysis}.  The popular reference \cite[Theorem 2.1.15]{nesterov2018lectures} establishes \eqref{eq:gradientlinearconvergence} with the more conservative rate
\begin{equation} \label{eq:nesterovrate}
 \bar \rho_\GD \coloneqq \left( 1-\frac{2\alpha mL}{L+m}\right)^{1/2}. 
\end{equation}
This quantity is never less than $\rho_\GD$, and equality holds if and only if $\alpha = 2/(L+m)$.

The rate $\rho_\GD$ is tight over $\mathcal{S}(m,L)$. Specifically, $\mathcal{S}(m,L)$ includes all the quadratic functions 
\[
f(x) = \tfrac{1}{2}\langle Qx,x \rangle + \langle p,x\rangle + b, \quad mI_n \preceq Q \preceq LI_n
\]
for which the system \eqref{eq:gradient} is linear.
Consequently, a spectral analysis shows that any $\rho > 0$ for which \eqref{eq:gradientlinearconvergence} holds is bounded from below by $\rho_\GD$. Since this lower bound is attained, the worst-case convergence rate of gradient descent over $\mathcal{S}(m,L)$ is determined purely by its performance on quadratic functions.\footnote{For the reader familiar with the control-theoretic absolute stability problem, the convergence result quoted above shows that gradient descent satisfies a version of the \emph{Aizerman conjecture}.}

Now consider the inexact version of \eqref{eq:gradient}, where the gradient is perturbed by an arbitrary sequence $e$ satisfying \eqref{eq:errormodel}:
\begin{equation} \label{eq:inexactgradient}
x_{t+1} = x_t - \alpha(\nabla f(x_t) + e_t).
\end{equation}
 If an estimate of the form \eqref{eq:gradientlinearconvergence} holds uniformly over $\mathcal{S}(m,L)$ for some $\rho>0$, then $\rho \geq \rho_{\GD}(\delta)$, where now
\[
\rho_\GD(\delta) = \rho_\GD(\delta; m,L,\alpha) \coloneqq \max(1-(1-\delta)\alpha m, (1+\delta)\alpha L -1).
\]
If $\alpha < 2/((1+\delta)L)$, then $\rho_\GD(\delta) \in (0,1)$.

 In contrast with the exact setting, we are not able to verify that the rate $\rho_\GD(\delta)$ is attained over $\mathcal{S}(m,L)$ for all choices of parameters and noise levels $\delta \in [0,1)$. Instead, we have the following:

\begin{prop} \label{prop:noisyS}
	Let $0< m < L$ and $\alpha > 0$. Suppose that $0 \leq \delta <2/(\kappa+1)$, and define
	\[
	 \alpha_- \coloneqq \frac{1}{1-\delta}\left( \frac{2}{L+m} - \frac{\delta}{m} \right), \quad \alpha_+ \coloneqq  \frac{1}{1+\delta}\left( \frac{2}{L+m} + \frac{\delta}{L} \right).
	\]
    If $\alpha \leq \alpha_-$ or $\alpha \geq \alpha_+$,
	then for any $f \in \mathcal{S}(m,L)$ the iterates \eqref{eq:inexactgradient} satisfy
	\[
	\| x_t-x_\star\| \leq \rho_\GD(\delta)^t \|x_0 - x_\star\|, \quad t \geq 0.
	\]
\end{prop}

It is important to note that the range of parameters for which Proposition \ref{prop:noisyS} holds is optimal with respect to the stability criterion used in its proof, namely a perturbed version of the circle criterion --- see Lemma \ref{lemm:perturbedFDIsector}. More precisely, the frequency test used to establish Proposition \ref{prop:noisyS} depends on multipliers encoding the constraints \eqref{eq:errormodel} and \eqref{eq:sector}. The corresponding test is feasible for $\rho = \rho_\GD(\delta)$ if and only if $(m,L,\alpha,\delta)$ satisfy the hypotheses of Proposition \ref{prop:noisyS}. In principle the test itself could be conservative (in terms of certifying asymptotic stability), but for a non-asymptotic tightness result see Lemma \ref{lemm:optimal} below.

When $\alpha \leq \alpha_-$, Proposition \ref{prop:noisyS} was previously established in \cite{de2017worst1} using a numerically-assisted method, as mentioned in \S \ref{subsect:overview}. Their result was stated for strongly convex functions, but it is clear that the proof also applies to $\mathcal{S}(m,L)$.

As $\kappa$ grows, the amount of noise that can be tolerated decreases towards zero. The optimal step size in the context of Proposition \ref{prop:noisyS} occurs when either $\alpha = \alpha_-$ or $\alpha = \alpha_+$. The corresponding convergence rate is
\begin{equation*} \label{eq:weakrate}
\rho = \frac{\kappa-1}{\kappa + 1} + \delta.
\end{equation*}
Meanwhile, the value of $\alpha$ that minimizes $\rho_\GD(\delta)$ over all $\alpha > 0$ and the corresponding value of $\rho_\GD(\delta)$ are given by
\begin{equation} \label{eq:sharprate}
	\alpha_\star \coloneqq \frac{2}{(1+\delta)L + (1-\delta)m}, \quad \rho_\GD(\delta; m,L,\alpha_\star) = \frac{ (1+\delta)\kappa -(1-\delta)}{(1+\delta)\kappa + (1-\delta)}.
\end{equation}
Observe that $\alpha_\star \in (\alpha_-, \alpha_+)$ for any $L > m$ and $\delta \in (0,1)$. Thus Proposition \ref{prop:noisyS} fails to certify the rate in \eqref{eq:sharprate}, even for arbitrarily small $\delta > 0$.
Instead, we can use the following result, which is new:

\begin{prop} \label{prop:weakGD}
		Let $0< m < L$ and $0 \leq \delta < 2\sqrt{\kappa}/(\kappa+1)$.  Define
		\begin{equation} \label{eq:generalrho}
			\bar  \rho_\GD(\delta) \coloneqq \left( 1-\frac{2 \alpha  L m}{L+m} + \frac{\alpha  \delta ^2 (L + m -2 \alpha Lm )}{2- \alpha  (L+m)}\right)^{1/2},
		\end{equation}
	where in order to ensure that $\bar \rho_\GD(\delta) \in (0,1)$ we assume that
		\[
		0 < \alpha <  \bar \alpha_- \coloneqq\frac{4 \kappa - \delta^2 (\kappa + 1)^2}{2m \kappa (1+\kappa)(1-\delta^2)}.
		\]
		Then for any $f \in \mathcal{S}(m,L)$ the iterates \eqref{eq:inexactgradient} satisfy
		\[
		\| x_t -x_\star \| \leq \bar\rho_\GD(\delta)^t \| x_0 - x_\star \|, \quad t \geq 0. 
		\]
\end{prop}

 It is always the case that $\bar \alpha_- \leq 2/((1+\delta) L)$, but $\alpha_+ > 2/((1+\delta) L)$ if $\delta > 2/(\kappa+1)$. Observe that the restriction on $\delta$ is weaker than the one in Proposition \ref{prop:noisyS}. The rate $\bar{\rho}_\GD(\delta)$ has the following precise characterization: if $(m,L,\alpha,\delta)$ are such that the inexact circle criterion is infeasible with $\rho_\GD(\delta)$, then $\bar{\rho}_\GD(\delta)$ is the smallest rate for which feasibility \emph{does} hold. Conversely, 
\[
\bar{\rho}_\GD(\delta) \geq \rho_\GD(\delta),
\] 
with equality if and only if $\alpha = \alpha_-$ or $\alpha = \alpha_+$. Note that when $\delta$ vanishes, $\bar {\rho}_\GD(0)$ is simply \eqref{eq:nesterovrate}.

Let us now assume that $\delta < 2\sqrt{\kappa}/(\kappa+1)$. The best possible rate afforded by either Proposition \ref{prop:noisyS} or Proposition \ref{prop:weakGD} is obtained by minimizing $\bar \rho_\GD(\delta)$ over $\alpha < \bar \alpha_-$. The minimizer $\bar \alpha_\star$ is
\begin{equation} \label{eq:baralphastar}
\bar \alpha_\star \coloneqq \frac{2}{L+m} - \frac{\delta (\kappa-1)}{m(1-\delta^2)^{1/2}\kappa^{1/2} (\kappa+1)},
\end{equation}
and the corresponding minimal rate is given by
\[
\bar\rho_\GD(\delta; m, L, \bar \alpha_\star) = \frac{1}{\kappa+1} \left((\kappa-1)((1-\delta^2)(\kappa-1)+4\kappa^{1/2}(1-\delta^2)^{1/2}\delta) + 4\kappa\delta^2\right)^{1/2}.
\]
Proposition \ref{prop:noisyS} and Proposition \ref{prop:weakGD} are best understood in tandem graphically --- see Figure \ref{fig:gd_sector}.

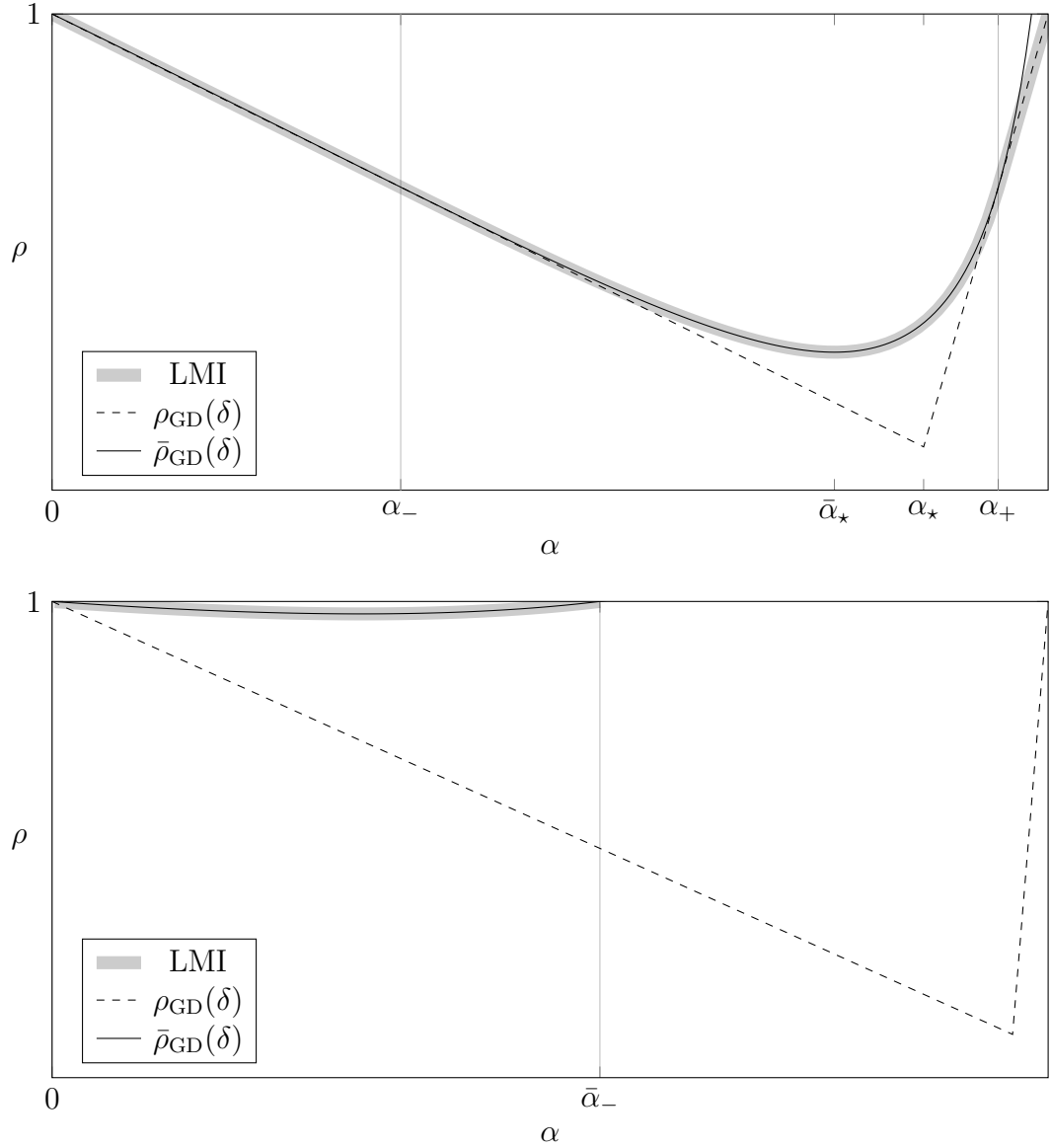
\begin{figure}
	\begin{center}
		\begin{tikzpicture}[
		    declare function ={ barrhogd(\L,\m,\d,\a) = sqrt(1-(2*\a*\L*\m)/(\L+\m)+\a*\d^2*(\L+\m-2*\a*\L*\m)/(2-\a*(\L+\m)));
		    },
	  	  declare function ={ alphaminus(\L,\m,\d) = (1/(1-\d))*(2/(\L+\m)-\d/\m);
	  	  },
    		declare function ={ alphaplus(\L,\m,\d) = (1/(1+\d))*(2/(\L+\m)+\d/\L);
  			  },
 			declare function ={ rhomin(\L,\m,\d) = ((1+\d)*\L-(1-\d)*\m)/((1+\d)*\L+(1-\d)*\m);
		},
			]
			\pgfmathsetmacro\alphaminusticksmall{alphaminus(3,1,.4)}
			\pgfmathsetmacro\alphaplusticksmall{alphaplus(3,1,.4)}
			\pgfmathsetmacro\rhomincoordsmall{rhomin(3,1,.4)}
				\pgfmathsetmacro\alphaminustickbig{alphaminus(3,1,.4)}
			\pgfmathsetmacro\alphaplustickbig{alphaplus(3,1,.4)}
			\pgfmathsetmacro\rhomincoordbig{rhomin(3,1,.4)}
			\begin{axis}[
				name=a,height=8cm, width=15cm, xlabel=$\alpha$, ylabel=${\rho}$, legend pos = south west,
				ylabel near ticks, 
				axis line style={-},
				xmin = 0,
				xmax = 0.47619,
				ymax = 1,
				xlabel style={yshift=0cm},
				ylabel style={rotate=-90, xshift=.3cm},
				xtick={0, 0.416667, 0.374012},
				xticklabels={$0$,  $\alpha_\star$, $\bar\alpha_\star$},
				extra x ticks = {\alphaminusticksmall, \alphaplusticksmall},
				extra x tick labels = {$\alpha_-$, $\alpha_+$},
				ytick={1},
				extra tick style={
					grid=major, }
				]
				\addplot[color=black, opacity=.2, line width=5pt] table[col sep=comma,header=false,x index=0,y index=1] {gd_inexact_smalldelta_sector.csv};
				
				\addplot+[color=black, mark = none, sharp plot, dashed] coordinates{(0,1) (0.416667,\rhomincoordsmall) (0.47619,1)};					
				
				\addplot[color=black, domain = 0 : .48, samples=200] {barrhogd(3, 1, .4, x)};
				\legend{LMI, $\rho_\GD(\delta)$, $\bar\rho_\GD(\delta)$}
			\end{axis}

				\begin{axis}[
			name=b, height=8cm, width=15cm, xlabel=$\alpha$, ylabel=${\rho}$, legend pos = south west,
			anchor=north west, at = (a.south west),
			yshift=-1.5cm,
			ylabel near ticks, 
			axis line style={-},
			xmin = 0,
			xmax = 0.37037,
			ymax = 1,
			xlabel style={yshift=0cm},
			ylabel style={rotate=-90, xshift=.3cm},
			xtick={0, 0.203704}, 
			xticklabels={$0$, $\bar\alpha_-$}, 
			ytick={1},
			xmajorgrids=true
			]
			\addplot[color=black, opacity=.2, line width=5pt] table[col sep=comma,header=false,x index=0,y index=1] {gd_inexact_largedelta_sector.csv};
			
			\addplot+[color=black, mark = none, sharp plot, dashed] coordinates{(0,1) (0.357143,\rhomincoordbig) (0.37037,1)};					
			
			\addplot[color=black, domain = 0 : .48, samples=200] {barrhogd(3, 1, .8, x)};
			\legend{LMI, $\rho_\GD(\delta)$, $\bar\rho_\GD(\delta)$}
		\end{axis}

		\end{tikzpicture}
	\end{center}
	\caption{ An illustration of Proposition \ref{prop:noisyS} and Proposition \ref{prop:weakGD}. These figures were produced with $m=1$ and $L = 3$. In the top figure $\delta = 2/5$, so that $\delta < 2/(\kappa +1)$, and in the bottom figure $\delta = 4/5$, so that $2/(\kappa +1) < \delta < 2\sqrt{\kappa}/(\kappa +1)$. The largest value on the $\alpha$ axis is $2/((1+\delta)L)$. The thick gray line is the optimal rate computed numerically using \eqref{eq:LMI} and the sector IQC (see \S  \ref{subsect:GDSml} for details). In the top figure, when $\alpha \leq \alpha_-$ or $\alpha \geq \alpha_+$ the numerically computed rate agrees with $\rho_\GD(\delta)$ in accordance with Proposition \ref{prop:noisyS}; if $\alpha \in (\alpha_-, \alpha_+)$, then it agrees with $\bar \rho_\GD(\delta)$. In the bottom figure only Proposition \ref{prop:weakGD} is applicable.	\label{fig:gd_sector}}
\end{figure}

\begin{figure}
	\begin{center}
		\begin{tikzpicture}[
			declare function ={ barrhogd(\L,\m,\d,\a) = sqrt(1-(2*\a*\L*\m)/(\L+\m)+\a*\d^2*(\L+\m-2*\a*\L*\m)/(2-\a*(\L+\m)));
			},
			declare function ={ alphaminus(\L,\m,\d) = (1/(1-\d))*(2/(\L+\m)-\d/\m);
			},
			declare function ={ alphaplus(\L,\m,\d) = (1/(1+\d))*(2/(\L+\m)+\d/\L);
			},
			declare function ={ rhomin(\L,\m,\d) = ((1+\d)*\L-(1-\d)*\m)/((1+\d)*\L+(1-\d)*\m);
			},
			]
			\pgfmathsetmacro\alphaminustick{alphaminus(3,1,.6)}
			\pgfmathsetmacro\alphaplustick{alphaplus(3,1,.6)}
			\pgfmathsetmacro\rhomincoord{rhomin(3,1,.6)}
			\begin{axis}[
				height=8cm, width=15cm, xlabel=$\alpha$, ylabel=${\rho}$, legend pos = south west,
				ylabel near ticks, 
				axis line style={-},
				xmin = 0,
				xmax = 0.416667,
				ymax = 1,
				xlabel style={yshift=0cm},
				ylabel style={rotate=-90, xshift=.3cm},
				xtick={0}, 
				xticklabels={$0$},
				ytick={1},
				xmajorgrids=true
				]
				
				\addplot[color=black, opacity=.2, line width=5pt] table[col sep=comma,header=false,x index=0,y index=1] {gd_inexact_convex.csv};
				
				\addplot+[color=black, mark = none, sharp plot, dashed] coordinates{(0,1) (0.384615,\rhomincoord) (0.416667,1)};					
			
				\legend{LMI, $\rho_\GD(\delta)$, $\bar\rho_\GD(\delta)$}
			\end{axis}
		\end{tikzpicture}
	\end{center}
	\caption{Here $m =1, \, L= 3$, and $\delta = 3/5$. The thick gray line is the optimal rate computed numerically using \eqref{eq:LMI} and the off-by-one IQC for strongly convex functions (see  \S \ref{subsect:inexactFmL} for details). These numerical results are consistent with Proposition \ref{prop:GDM}. \label{fig:gd_convex}}
\end{figure}
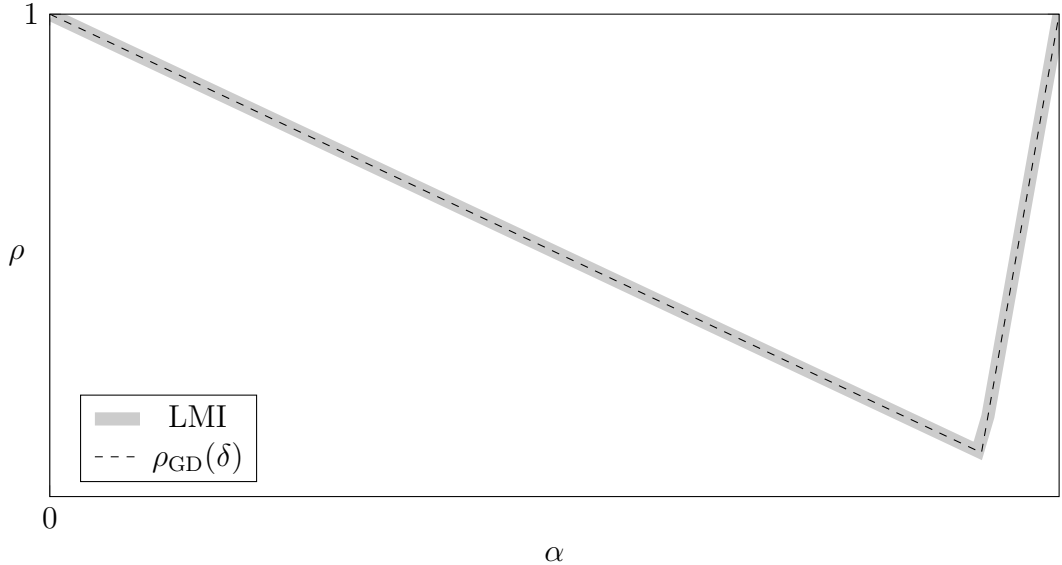

Finally, we show that by restricting to $\mathcal{F}(m,L)$, inexact gradient descent attains its optimal exponential convergence rate $\rho_\GD(\delta)$ for any $\delta \in [0,1)$.\footnote{In the language of the absolute stability problem, inexact gradient descent satisfies a version of the \emph{Kalman conjecture}.} 

\begin{prop} \label{prop:GDM}
	Let $0< m<L$ and $0 \leq \delta <1$. If $0 < \alpha < 2/((1+\delta) L) $ and $\rho \geq \rho_\GD(\delta)$, then there exists $c>0$ such that for any $f \in \mathcal{F}(m,L)$ the iterates \eqref{eq:gradient} satisfy
	\[
	 \| x_t -x_\star\| \leq c\rho^t \| x_0 - x_\star\|, \quad t \geq 0.
	\]
\end{prop}
In particular, the optimal rate \eqref{eq:sharprate} is feasible for inexact gradient descent on strongly convex functions --- see Figure \ref{fig:gd_convex}. When an exact line search is employed, it was shown in \cite{de2017worst} that the rate \eqref{eq:sharprate} holds for inexact gradient descent on strongly convex functions. 

\begin{rema}  As observed by an anonymous referee, Proposition \ref{prop:GDM} can already be deduced from \cite[Appendix B, Proof of Theorem 5.3]{de2017worst1}. There, the authors show that
\begin{equation} \label{eq:flyapunov}
f(x_{t+1}) - f(x_\star) \leq \rho_\GD(\delta)^2 (f(x_t) - f(x_\star)), \quad t \geq 0
\end{equation}
whenever $f \in \mathcal{F}(m,L)$ and $\alpha \leq \alpha_+$, but a close examination shows that the same proof actually applies for $\alpha \leq 2/((1+\delta)L)$.
\end{rema} 

Using the convex interpolation and duality techniques from \cite{taylor2017smooth}, one can also show that the above results for $\mathcal{S}(m, L)$ are optimal over one time step in the following non-asymptotic sense, at least in dimensions $n \geq 3$: 

\begin{lemm} \label{lemm:optimal}
 Let $0 < m < L$ and $0 \leq \delta < 2\sqrt{\kappa}/(\kappa+1)$. Suppose that $n \geq 3$ and $\alpha \in (0, \bar \alpha_-)$. If
\begin{equation} \label{eq:rhoupperbound}
\rho < \begin{cases} \rho_\GD(\delta) \text{ if } \alpha \notin (\alpha_-, \alpha_+), \\
	\bar \rho_\GD(\delta) \text{ if } \alpha \in (\alpha_-, \alpha_+),
	\end{cases}
\end{equation}
 then there exists $f \in \mathcal{S}(m,L)$ and $x_0, e_0 \in \RR^{n}$ such that
\[
\| x_0 - \alpha(\nabla f(x_0) + e_0) - x_\star \| > \rho \| x_0 - x_\star\|, \quad \| e_0 \| \leq \delta \| \nabla f(x_0) \|.
\]
\end{lemm}
In fact, we can choose a counterexample $f \in \mathcal{F}(m,L)$, which does not contradict the asymptotic result of Proposition \ref{prop:GDM} (because of the constant $c>0$ appearing in the stability estimate), nor \eqref{eq:flyapunov} (because convergence to optimality is not being measured in the Euclidean norm). Such non-asymptotic optimality results are the subject of the closely related \emph{performance estimation problem} discussed in \S \ref{subsect:previouswork} below.

\subsection{Results for the inexact Triple Momentum Method} \label{subsect:TMMresults}
A general class of momentum-based methods is given by the recursion
\begin{equation} \label{eq:degree2} 
\begin{aligned}
\xi_{t+1} &= \xi_t + \beta(\xi_t - \xi_{t-1}) - \alpha \nabla f(y_t), \\
y_t &= \xi_t + \gamma(\xi_t - \xi_{t-1}).
\end{aligned}
\end{equation} 
In particular, these iterates admit a state-space representation as in \eqref{eq:lure}, where the state is $(\xi_t, \xi_{t-1}) \in \RR^{2n}$. For simplicity, we assume the initialization $\xi_{-1} = 0$. Define
\[
\rho_\TM \coloneqq  1-1/\sqrt{\kappa}.
\]
The Triple Momentum Method (TMM), developed in \cite{van2017fastest}\footnote{In fact, the choice of parameters \eqref{eq:TMparameters} was originally motivated by a frequency-domain analysis; for more, see \S \ref{sect:TMM} and \cite{van2017fastest}.}, is given by the choices
\begin{equation} \label{eq:TMparameters}
\alpha = \frac{1+\rho_\TM}{L}, \quad \beta = \frac{\rho_\TM^2}{2-\rho_\TM}, \quad \gamma = \frac{\rho^2_\TM}{(1+\rho_\TM)(2-\rho_\TM)}.
\end{equation}
TMM is the fastest known gradient-based method for optimizing strongly convex functions: given $\rho \geq \rho_\TM$, there exists $c>0$ such that
\[
\| \xi_t - \xi_\star \| \leq c \rho^t  \|\xi_0 - \xi_\star\|, \quad t \geq 0
\]
whenever $f \in \mathcal{F}(m,L)$, where $\xi_\star$ is the unique minimizer of $f$.

Next, suppose that TMM is perturbed by replacing $\nabla f(y_t)$ with $\nabla f(y_t) + e_t$ in \eqref{eq:degree2}, where the error sequence $e$ satisfies \eqref{eq:errormodel}. We refer to the resulting recursion as inexact TMM.
\begin{prop} \label{prop:TMM}
 Let $0 < m < L$, and define $\theta_\TM \coloneqq (2-\rho_\TM)\rho_\TM + (2+\rho_\TM)\delta$. Furthermore, define
 \begin{equation} \label{eq:TMMperturbedrate}
 \bar \rho_\TM(\delta) \coloneqq \frac{\theta_\TM + (\theta_\TM^2 + 4\delta \rho_\TM^2(2-\rho_\TM))^{1/2}}{2(2-\rho_\TM)},
 \end{equation}
 where in order to ensure $\bar \rho_\TM(\delta) \in (0,1)$ we assume that
  \[
 0 \leq \delta < \frac{\sqrt{\kappa}+1}{4\kappa - 3\sqrt{\kappa} +1}.
 \]
If $\rho \geq \bar \rho_\TM(\delta)$, then there exists $c>0$ such that for any $f \in \mathcal{F}(m,L)$ the inexact TMM iterates satisfy
	\[
	\| \xi_t -\xi_\star\| \leq c\rho^t \| \xi_0 - \xi_\star\|, \quad t \geq 0.
	\]
\end{prop} 

Proposition \ref{prop:TMM} is based on a perturbed version of the Jury--Lee criterion --- see Lemma \ref{lemm:perturbedFDIconvex}. We conjecture that the rate \eqref{eq:TMMperturbedrate} is actually \emph{optimal} with respect to this criterion. While a proof is currently out of reach, we have been unable to find any counterexamples (see Figure \ref{fig:tmm_inexact} for some numerical illustrations).

\subsection{Previous work} \label{subsect:previouswork}

Following \cite{lessard2016analysis}, numerous papers have adopted a   control-theoretic viewpoint for the analysis of gradient-based optimization methods; a partial list of recent contributions in a variety of directions includes \cite{hu2016exponential,hu2017dissipativity,cyrus2018robust,van2017fastest,taylor2018lyapunov,hu2018dissipativity,lessard2019direct,fazlyab2018analysis,hassan2019proximal,michalowsky2019robust,badithela2019analysis,nishihara2015general,hu2019characterizing,francca2016explicit,hu2017analysis,drummond2018accelerated,michalowsky2020design,gramlich2020convex,xiong2020analytical,hu2020analysis,hu2017control,safavi2018explicit}. A few of these works apply frequency methods, notably \cite{van2017fastest,cyrus2018robust,xiong2020analytical}. However, the frequency approach is not always taken as far as possible. For instance, \cite{van2017fastest,cyrus2018robust} use a frequency interpretation for algorithm \emph{design} via pole-zero assignment, but then resort to time-domain methods in order to \emph{prove} their convergence results. This is because the frequency criteria that are employed degenerate at the margins of Lyapunov stability.  In this paper we emphasize a more robust set of tools, in particular Popov's hyperstability theorem, that enables a self-contained analysis in the frequency domain; for a thorough discussion, see \S \ref{subsect:feedback}.

Earlier, \cite{drori2014performance} introduced a closely related semidefinite programming formulation, the so-called performance estimation problem (PEP), for analyzing the worst-case performance of iterative optimization methods over a given number of time steps. Subsequent developments include \cite{de2017worst,de2017worst1,taylor2017exact,kim2016optimized,ryu2018operator,drori2018efficient,taylor2018exact,kim2017convergence}. Importantly, one input to PEP is a choice of user-specified \emph{performance measure}. An example is distance to optimality, but this is not always appropriate, e.g., for Nesterov's method or TMM (of course distance to optimality is a good measure of \emph{convergence}, but not necessarily for the behavior of an algorithm over a given set of time steps). 

In general, finding a suitable performance measure is nontrivial. In \cite{taylor2018lyapunov}, the authors search for a performance measure in the form of a Lur'e--Postnikov-type Lyapunov function; under a dimensionality condition, they formulate a semidefinite program whose feasibility is equivalent to the existence of such a Lyapunov function decreasing geometrically at each iteration over $\mathcal{F}(m,L)$. Of course it is not clear whether this approach provides necessary conditions for exponential stability, since the class of Lyapunov functions under consideration is parametric.

The specific IQCs used in this paper, as well as other works derivative of \cite{lessard2016analysis}, provide simpler sufficient  conditions for the existence of a Lyapunov function within a certain parametric family. However, we do not stress the existence of a Lyapunov function, and instead choose to work within Popov's framework of hyperstability (the relationship between the two viewpoints is especially clear in the frequency domain --- see the discussion following Lemma \ref{lemm:jurylee}).

As noted in \S \ref{subsect:main}, some analytic bounds for the convergence of inexact gradient descent under the error model \eqref{eq:errormodel} were given in \cite{de2017worst,de2017worst1}. Numerical computations of exponential convergence rates under the same error model were carried out in \cite{lessard2016analysis,cyrus2018robust} for the case of strongly convex functions, but without analytic results. 
Simpler results can be found in
\cite{polyak1987introduction,bertsekas1999nonlinear}. Stochastic gradient descent, where the gradients are subject to additional relative noise, was analyzed in \cite{hu2020analysis}. The relative error model also arises naturally in the analysis of various incremental gradient methods for minimizing finite sums \cite{friedlander2012hybrid} (inexactness arises from sampling only a subset of the terms in the sum).

Other error models are considered in \cite{d2008smooth,devolder2014first,devolder2013first,devolder2013intermediate,cohen2018acceleration}, with both deterministic and stochastic perturbations. For example, \cite{devolder2014first} considers an inexact oracle for $\mathcal{F}(0,L)$: instead of $(f(y), \nabla f(y))$, the oracle returns a pair $(f_\delta(y), g_\delta(y))$ satisfying
\[
0 \leq f(x) - (f_\delta(y) + \langle g_\delta(y), x-y\rangle ) \leq \tfrac{L}{2} \|x-y\|^2, \quad x \in \RR^n.
\]
The case of strongly convex functions was considered in \cite{devolder2013first}. 
In contrast, the error model in this paper only involves errors in the gradient calculations, and the aforementioned works do not overlap with ours.

\subsection{Outline of the paper}

In \S \ref{sect:stability}, we introduce Popov's notion of hyperstability and its frequency-domain formulation. The two frequency criteria used in this paper, the circle and Jury--Lee criteria, are described in \S \ref{subsect:sector} and \S \ref{subsect:ZF}. Applications of these two criteria in the exact setting to gradient descent and Nesterov's method are discussed in \S \ref{subsect:GD} and \S \ref{subsect:nesterov}, respectively.

The necessary modifications to the frequency criteria in the inexact setting are described in \S \ref{subsect:noise}. The results for inexact gradient descent given in \S \ref{subsect:main} are proved in \S \ref{sect:inexactGD}. Proposition \ref{prop:TMM} for inexact TMM is derived in \S \ref{sect:TMM}.

\subsection{Notation}
Given a matrix $M \in \CC^{p\times q}$, we use the notation $M^\top$ for its transpose and $M^*$ for its adjoint. If $M$ is square, we define $\Re M \coloneqq (M + M^*)/2$. Positive definiteness and semidefiniteness is denoted $M \succ 0$ and $M \succeq 0$. The standard sesquilinear (in the second factor) inner product on $\CC^p$ (and its restriction to $\RR^p$) is denoted by $\left < \cdot, \cdot \right>$. For an arbitrary complex-valued function $f$ we write $f \equiv 0$ if $f$ vanishes identically on its domain. Denote by $\sig_d$ the space of sequences on $\mathbb{N}_{\geq 0}$ with values in $\CC^d$.

\section{Stability criteria} \label{sect:stability}

\subsection{Hyperstability} \label{subsect:feedback}
 Consider a linear time-invariant (LTI) control system of the form
\begin{equation} \label{eq:dynamicalsystem}
x_{t+1} = Ax_t + Bu_t,
\end{equation}
where $A \in \RR^{d\times d}$ and $B\in \RR^{d \times n}$. Also let
\[
\sol(A,B) = \{(x,u) \in \ell_{d+n} : \, x_{t+1} = Ax_t + Bu_t\}.
\]
(The notation $\sol$ refers to a \emph{behavior} in the sense of Willems \cite{willems1997introduction}.)
Popov formulated the notion of hyperstability in the early 1960's to formalize the problem of certifying stability for \eqref{eq:dynamicalsystem} when the control $u \in \ell_n$ and state $x \in \ell_d$ are jointly constrained by quadratic inequalities. 

First we discuss the notion of an integral quadratic constraint (IQC). Fix a quadratic form
\begin{equation} \label{eq:QSR}
\sigma(x,u) =  \begin{bmatrix}
	x \\ u
\end{bmatrix}^* \begin{bmatrix}
Q & S^\top \\
S & R
\end{bmatrix} \begin{bmatrix}
x \\ u
\end{bmatrix}
\end{equation}
on $\CC^{d+n}$, where $Q =Q^\top \in \RR^{d\times d},\,S \in \RR^{n \times d}$, and $R = R^\top \in \RR^{n\times n}$.

\begin{defi}
	Let $\rho > 0$. A pair $(x,u) \in \sig_{d+n}$ is said to satisfy the $\rho$-IQC defined by $\sigma$ if
	\begin{equation} \label{eq:IQC}
	\sum_{t=0}^T \rho^{-2t} \sigma(x_t, u_t) \geq 0  \text{ for all }T\geq 0.
	\end{equation}
	Let $\IQC(\sigma; \rho)$ denote the set of all such pairs. Furthermore, given $\sigma_1, \ldots, \sigma_q$, define
	\[
	\IQC(\sigma_1, \ldots, \sigma_q; \rho) := \bigcap_{j = 1}^q \IQC(\sigma_j; \rho).
	\]
\end{defi}
Although dating back to early work of Yakubovich and Popov, the term IQC was popularized in the landmark paper \cite{megretski1997system}. While we use the IQC terminology, it is important to note that we do not apply the frequency-domain formulation in \cite{megretski1997system} to the analysis of gradient methods; instead, we depend on an earlier result due to Popov.

\begin{defi}
	Given $\sigma_1,\ldots, \sigma_q$ and $\rho > 0$, we say that $(A,B; \sigma_1,\ldots,\sigma_q)$ is $\rho$-hyperstable if there exists $c>0$ such that 
	\begin{equation} \label{eq:hyperstableestimate}
	\|x_t\| \leq c\rho^{t} \|x_0\|, \quad t \geq 0
	\end{equation}
	whenever $(x,u) \in \sol(A,B) \cap \IQC(\sigma_1,\ldots,\sigma_q; \rho)$.
\end{defi}
In other words, hyperstability guarantees that the system $x_{t+1} = Ax_t + Bu_t$ is uniformly exponentially stable over arbitrary input sequences $u$ as long as $(x,u) \in \IQC(\sigma_j;\rho)$ for each $j = 1,\ldots,q$. 

To show that $(A,B;\sigma)$ is $\rho$-hyperstable for a single constraint $\sigma(x,u) =  \langle Qx,x\rangle + 2\Re \langle Sx,u \rangle + \langle Ru,u\rangle$, consider the LMI
\begin{equation} \label{eq:LMI}
\begin{bmatrix} \tag{LMI}
A^\top PA - \rho^2 P & A^\top P B\\
B^\top P A & B^\top P B
\end{bmatrix} + \begin{bmatrix}
Q & S^\top \\ 
S & R
\end{bmatrix} \preceq 0
\end{equation}
with decision variable $P = P^\top \in \RR^{d \times d}$. In \cite{willems1972dissipativeI,willems1972dissipativeII}, \eqref{eq:LMI} is called the  \emph{dissipation inequality}, and the function
\[
x \mapsto \langle Px, x\rangle 
\]
is called a \emph{storage function}. If \eqref{eq:LMI} holds, then each $(x,u) \in \sol(A,B) \cap \IQC(\sigma;\rho)$ satisfies
\begin{equation} \label{eq:integrateddissipation}
\langle Px_t, x_t \rangle \leq \rho^{2t} \langle Px_0, x_0\rangle, \quad t \geq 0.
\end{equation}
Hyperstability is an immediate consequence of \eqref{eq:integrateddissipation} provided $P \succ 0$. 
For fixed $(A,B; \sigma)$ and $\rho > 0$, feasibility of \eqref{eq:LMI} for $P \succeq0$ is a convex program, and can be solved numerically (see \cite{lessard2016analysis} for more details).

The dissipation inequality can also provide sufficient conditions for $\rho$-hyperstability in the case of multiple constraints by observing that
\[
\IQC(\sigma_1,\ldots,\sigma_q;\rho) \subset \IQC(\lambda_1 \sigma_1 + \cdots + \lambda_q \sigma_q; \rho)
\]
for any collection of nonnegative multipliers $\lambda_1, \ldots \lambda_q$. Taking $\sigma = \sum \lambda_j \sigma_j$, one can also consider feasibility of \eqref{eq:LMI} with $\lambda_1, \ldots, \lambda_p$ as additional decision variables. Observe that this formulation is also a convex program.

 Popov gave a complete frequency-domain characterization of hyperstability under a certain \emph{minimal stability} hypothesis, which we now discuss. If $\rho = 1$, then hyperstability implies that a trajectory $x$ is bounded whenever the input is such that $(x,u) \in  \IQC(\sigma;1)$.  Minimal stability requires us to exhibit an input for which the corresponding trajectory $x$ actually converges to zero and $(x,u) \in \IQC(\sigma;1)$. 
The condition for general $\rho > 0$ is just a rescaling of this requirement:

\begin{defi} \label{defi:minimal}
	The triple $(A,B; \sigma)$ is said to be $\rho$-minimally stable if for every $x^0 \in \CC^d$ there exists $(x,u) \in \sol(A,B) \cap \IQC(\sigma; \rho)$ satisfying
	\[
	x_0 = x^0, \quad \rho^{-t} x_t \rightarrow0 \text{ as } t \rightarrow \infty.
	\]
\end{defi}

Among other consequences, minimal stability implies that every storage function is positive semidefinite. Indeed, given an arbitrary $x^0 \in \RR^d$, let $(x,u) \in \sol(A,B)$ be a trajectory as in the definition of minimal stability; let $t \rightarrow \infty$ in \eqref{eq:integrateddissipation} to deduce that 
\[
\langle Px^0,x^0\rangle \geq 0.
\]
 In applications the minimal stability hypothesis is typically  straightforward to verify --- see Lemma \ref{lemm:minimalstability} below.

 Next, given $(A,B)$ and $\sigma(x,u) =  \langle Qx,x\rangle + 2\Re \langle Sx,u \rangle + \langle Ru,u\rangle$, define the Popov function
\[
\Pi(\zeta,z) = \Pi(\zeta,z; A,B ; \sigma) := \begin{bmatrix}
(\conj \zeta I - A)^{-1}B \\ I_n
\end{bmatrix}^* \begin{bmatrix}
Q & S^\top \\
S & R
\end{bmatrix}\begin{bmatrix}
(z I-A)^{-1}B \\
I
\end{bmatrix},
\]
which is a meromorphic function of $(\zeta,z)$ with Hermitian values in $\CC^{n\times n}$. The dissipation inequality is linked to properties of $\Pi(\zeta,z)$ via the following observation: if $P = P^\top$ solves \eqref{eq:LMI}, then
	\begin{equation} \label{eq:FDI} \tag{FDI}
	\Pi(\bar z, z) \preceq 0 \text{ whenever $|z| =\rho$ and $\det(A-zI) \neq 0$.}
	\end{equation}
This is easily seen by multiplying \eqref{eq:LMI} on the right by the column vector $((zI-A)^{-1}Bu, u)$ and by its conjugate transpose on the left. The converse of this result is the content of the Kalman--Yakubovich--Popov (KYP) lemma (sometimes called the Kalman--Szeg\"{o} lemma in discrete time \cite{kalman1963stabilite}):
\begin{lemm} \label{lemm:KYP}
	If $(A,B)$ is controllable, then there exists $P = P^\top \in \RR^{d\times d}$ satisfying \eqref{eq:LMI} if and only if \eqref{eq:FDI} holds.
\end{lemm}

Recall that $(A,B)$ is controllable if for any pair of states there exists an input driving the system $x_{t+1} = Ax_t + Bu_t$ from the initial state to the terminal state in finite time. Equivalently, $(A,B)$ is controllable if
\begin{equation} \label{eq:controllable}
\mathrm{rank} \begin{bmatrix}
	B & AB & \cdots & A^{d-1}B
\end{bmatrix} = d.
\end{equation}
(See \cite{kalman2010lectures}.) Perhaps the simplest proof of Lemma \ref{lemm:KYP} is due to Willems via an optimal control argument \cite[Theorem 4]{willems1971least}.\footnote{The reader should ignore the statement about the finiteness of $V_\mathrm{f}$, which is famously false \cite{willems1974existence}.} For the classical proof using spectral factorizations, see \cite[\S 10, Theorem 1]{popov1973hyperstability}. An alternative argument is provided in \cite{rantzer1996kalman}.\footnote{\cite{rantzer1996kalman} explicitly states that $A$ should not have spectrum with $|z| = \rho$, but this is unnecessary. Such an assumption can always be removed when $(A,B)$ is controllable by shifting the eigenvalues off of $|z| = \rho$ via linear feedback: an equivalent characterization of controllability is that for any monic polynomial there is a matrix $N$ such that the characteristic polynomial of $A + BN$ is the given polynomial \cite[\S 34, Theorem 1]{popov1973hyperstability}. The corresponding linear change of variables $(x, u) \mapsto (x, u + Nx)$ does not affect the validity of \eqref{eq:FDI} (it corresponds to a congruency transformation of the Popov function by a rational matrix function), and leaves invariant any solution of \eqref{eq:LMI}.}

Even if $(A,B)$ is controllable and \eqref{eq:FDI} holds, Lemma \ref{lemm:KYP} does not necessarily imply feasibility of \eqref{eq:LMI} for $P \succ 0$. In particular, a naive application of \eqref{eq:integrateddissipation} does not guarantee hyperstability. Instead, Popov established the following result.

\begin{theo} [{\cite[\S 18, Theorem 1]{popov1973hyperstability}}] \label{theo:popov}
	Given $\rho > 0$, suppose that $(A,B; \sigma)$ satisfies the following conditions.
	\begin{enumerate} \itemsep6pt
		\item $B\neq 0$, \label{it:popov1}
		\item  $(A,B;\sigma)$ is $\rho$-minimally stable, \label {it:popov2}
		\item $(\zeta,z) \mapsto \det \Pi(\zeta,z) \not\equiv 0$. \label{it:popov3}
	\end{enumerate}
	Then $(A,B;\sigma)$ is $\rho$-hyperstable if and only if \eqref{eq:FDI} holds.
	
	 In addition, if $(A,B)$ is controllable and $\Pi(\bar z_0, z_0) \prec 0$ for some $|z_0| = \rho$, then $(A,B;\sigma)$ is $\rho$-hyperstable if and only if \eqref{eq:LMI} admits a solution $P \succ 0$.
\end{theo}

Apart from dispensing with issues like controllability, the main advantage of Theorem \ref{theo:popov} in the context of this paper is that it replaces the search for $P$ with the task of checking \eqref{eq:FDI}. Although in principle the latter consists of verifying an infinite set of inequalities, the fact that $\Pi(\bar z,z)$ involves only rational functions of $z$ reduces \eqref{eq:FDI} to a small number of scalar inequalities. This plays a crucial role in analytic proofs, since additional parameters tend to make such problems disproportionately harder. 

Theorem \ref{theo:popov} is also particularly well-suited for analysis at the margins of $\ell^\infty$ stability: \eqref{eq:FDI} need not hold strictly, and $A$ can have eigenvalues on the circle of radius $\rho$. This should be compared with other frequency-domain formulations that are often cited in the optimization literature based on \cite{megretski1997system} or strict versions of the KYP lemma (for example \cite{boczar2015exponential,van2017fastest}). Such results are useful in the context of $\ell^2$ stability, but fail to characterize $\ell^\infty$ behavior as \eqref{eq:FDI} degenerates.

Using only the assumption $B \neq 0$, it is easy to show that $\rho$-hyperstability implies \eqref{eq:FDI}: suppose that 
\[
\langle \Pi(\bar z_0,z_0)u^0, u^0\rangle > 0
\] 
for some  $\CC^n \ni u^0 \neq 0$, where $|z_0| > \rho$ not an eigenvalue of $A$. By perturbing $u^0$ and using $B \neq 0$, we can assume that $Bu^0 \neq 0$. If $u \in \ell_n$ and $x \in \ell_d$ are defined by 
\[
u_t = z_0^t u^0, \quad x_t = z_0^t (zI-A)^{-1} Bu^0,
\]
then $(x,u) \in \sol(A,B) \cap \IQC(\sigma;\rho)$. However, $|\rho^{-t} x_t| \rightarrow \infty$ as $t\rightarrow \infty$ since $Bu^0 \neq 0$, which contradicts $\rho$-hyperstability. Thus we must have $\Pi(\bar z_0,z_0) \preceq 0$, so \eqref{eq:LMI} holds by continuity up to $|z| = \rho$.

The converse proof that \eqref{eq:FDI} implies hyperstability under the hypotheses of Theorem \ref{theo:popov} is not entirely trivial; it relies on the existence of zero-dynamics for certain nondegenerate control systems which is outside the scope of this paper. Roughly speaking, the proof proceeds in three steps:
\begin{enumerate} \itemsep6pt
	\item Use that $B \neq 0$ to pass to a controllable subsystem; the uncontrollable states are estimated using minimal stability. 
	\item Apply the KYP lemma to obtain a solution $P = P^\top$ of \eqref{eq:LMI}, which is positive semidefinite by minimal stability. States in the positive spectral subspace of $P$ are estimated by \eqref{eq:integrateddissipation}.
	\item The nondegeneracy condition $ \det \Pi(\zeta,z) \not\equiv 0$ is used to construct a  normal form that allows for states in $\ker P$ to be estimated using minimal stability. 
\end{enumerate}
The original proof in \cite{popov1973hyperstability} is not entirely straightforward, but there do not seem to be many alternative references. A modern proof in the continuous-time setting is given in \cite{gannot2019frequency}; the adaptations needed in discrete time are straightforward.

The hypothesis that $\det \Pi(\zeta,z) \not\equiv 0$ in Theorem \ref{theo:popov} is not intuitive, but it is typically easy to check. For instance, it is always satisfied if 
\[
\det(S(z I - A)^{-1}B + R) \not\equiv 0,
\]
as can seen by taking $\zeta \rightarrow \infty$. This in turn is true if $\det R \neq 0$.

In the context of this paper, minimal stability is established via linear feedback, i.e., by choosing a control in the form $u = Nx$. Given $\rho > 0$, a square matrix is said to be $\rho$-Schur if all its eigenvalues lie in the open disk of radius $\rho$. The following result is used repeatedly.
\begin{lemm} \label{lemm:minimalstability}
	If there exists $N \in \CC^{n \times d}$ and $\varepsilon \in [0,2/\| R\|]$ such that $A+B(I+\varepsilon R)N + \varepsilon BS$ is $\rho$-Schur and 
	\begin{equation} \label{eq:minimalQSR}
	Q + 2 \Re N^* S + N^* RN \succeq 0,
	\end{equation}
	then $(A,B;\sigma)$ is $\rho$-minimally stable.
\end{lemm} 
\begin{proof}
	The inequality \eqref{eq:minimalQSR} guarantees that $(x,Nx)\in \IQC(\sigma;\rho)$ for any sequence $x \in \ell_d$.  Now observe that \eqref{eq:minimalQSR} is stable under replacing $N$ with $N + \varepsilon (S + RN)$, provided $\varepsilon\in [0,2/\|R\|]$. Here the norm of $R$ is the operator norm. The $\rho$-Schur condition implies that solutions of the closed-loop system  \[
	x_{t+1} = (A+B(N + \varepsilon(S + RN))x_t
	\] 
	satisfy $\rho^{-t}x_t \rightarrow 0$ as $t\rightarrow \infty$.
\end{proof}

The following result is useful in bounding the values of certain multipliers. For simplicity we only provide the (trivial) proof in the controllable case, since the full strength of the result is only used for gradient descent, whose underlying LTI system is evidently controllable. It is true in general by passing to a controllable subsystem as in \cite[Eq. (26), \S 14]{popov1973hyperstability};
\begin{lemm} \label{lemm:Rconstraint}
		Let $B \neq 0$. If $(A,B;\sigma)$ is $\rho$-minimally stable and $\rho$-hyperstable, then $R \preceq 0$. 
\end{lemm}
\begin{proof}
Since \eqref{eq:FDI} holds, it follows that in the controllable case \eqref{eq:LMI} admits a solution $P \succeq 0$. We then conclude from the bottom right entry of \eqref{eq:LMI} that $R \preceq 0$.
\end{proof}

When $(A,B)$ is controllable and the nondegeneracy condition $\Pi(\bar z_0, z_0) \prec 0$ holds for some $|z_0| = \rho$, Popov showed that the so-called maximal solution of the KYP equations (see \cite{willems1971least} for the definitions) is positive definite.\footnote{In general, this is not true of \emph{every} solution to \eqref{eq:LMI}.} Unfortunately, this nondegeneracy condition often fails at the margins of stability. Instead, we record a simple algebraic criterion that ensures every positive semidefinite solution of \eqref{eq:LMI} is actually positive definite. Recall that the (in this case real-valued) pair $(A,S)$ is said to be observable if $(A^\top, S^\top)$ is controllable. Equivalently, $(A,S)$ is observable if every $A$-invariant subspace of $\RR^d$ contained in $\ker S$ is trivial.
\begin{lemm} \label{lemm:observable}
Let $Q \succeq 0$, and suppose that $P \succeq 0$ solves \eqref{eq:LMI}. If $(A,S)$ is observable, then $\ker P = \{0\}$.
\end{lemm}
\begin{proof}
	It suffices to show that $\ker P$ is $A$-invariant and contained in $\ker S$.
	The fact that $\ker P$ is $A$-invariant follows immediately from the Lyapunov inequality 
	\[
	A^\top P A - \rho^2 P \preceq 0
	\]implied by the upper left block of \eqref{eq:LMI}. To see that $\ker P \subset \ker S$, fix $ x \in \RR^d$ and pair \eqref{eq:LMI} with $(x,\varepsilon Sx)$. Then
	\[
	\langle P(A+\varepsilon BS)x, (A+\varepsilon BS)x \rangle -\rho^2 \langle Px,x \rangle \leq -(2\varepsilon - \varepsilon^2 \| R \| ) \| Sx\|^2.
	\]
	If $Px = 0$ and we choose $\varepsilon \in (0,2/\| R \|)$, then $Sx = 0$.
\end{proof}

\subsection{Circle criterion} \label{subsect:sector}
Consider an LTI system of the form \eqref{eq:dynamicalsystem}, and let $C \in \RR^{n\times d}$. Given $m < L$, consider the so-called sector IQC defined by the quadratic form
\begin{equation} \label{eq:sectorform}
\sigma_0(x,u) = \sigma_0(x,u; m,L,C) \coloneqq  2\langle u - mCx, LCx - u \rangle.
\end{equation}
If $f \in \mathcal{S}_0(m,L)$, then 
\[
(x, \nabla f(Cx))\in \IQC(\sigma_0; \rho)
\] 
for every $\rho > 0$ and every real-valued $x \in \ell_d$. The corresponding Popov function is
\[
\Pi (\conj z,z) = -2\Re \left[ (I - mH(z))^*(I-LH(z)) \right],
\]
where $H(z) = C(zI-A)^{-1}B$.
\begin{lemm} [Circle criterion] \label{lemm:circle}
Let $\rho > 0$. If $B \neq 0$ and there exists $k \in [m,L]$ such that $A+kBC$ is $\rho$-Schur, then $(A,B;\sigma_0)$ is $\rho$-hyperstable if and only if 
\begin{equation} \label{eq:circleFDI}
\Re \left[ (I - mH(z))^*(I-LH(z)) \right]\succeq 0 
\end{equation}
whenever $|z| = \rho$ and $\det(A-zI) \neq 0$.
\end{lemm}
\begin{proof} 
If $A+ kBC$ is $\rho$-Schur for some $k \in [m,L]$, then $(A,B; \sigma_0)$ is $\rho$-minimally stable. Indeed, the feedback $u = mCx$ renders the left-hand side of \eqref{eq:minimalQSR} identically zero, at which point the result follows from Lemma \ref{lemm:minimalstability} with
\[
k = (1-\varepsilon)m + \varepsilon L,\quad \varepsilon \in [0,1].
\]
Since $R = -2I$ has nonzero determinant, Theorem \ref{theo:popov} applies.
\end{proof} 

If $(A,B;\sigma_0)$ is $\rho$-hyperstable for a given $\rho \in (0,1)$, then the system
\begin{equation} \label{eq:circletrajectory}
x_{t+1} = Ax_t + B\nabla f(Cx_t)
\end{equation}
is uniformly $\rho$-exponentially stable over $\mathcal{S}_0(m,L)$. More precisely, there exists $c>0$ such that
\[
\| x_t \| \leq c \rho^t \| x_0 \|, \quad t \geq 0
\] 
whenever $f \in \mathcal{S}_0(m,L)$.
Finally, suppose that the dissipation inequality \eqref{eq:LMI} corresponding to $(A,B;\sigma_0)$ admits a solution $P \succ 0$ for a given $\rho > 0$. Then
\[
\langle Px_{t+1}, x_{t+1} \rangle \leq \rho^2 \langle Px_t, x_t \rangle, \quad t \geq 0
\]
along any trajectory of \eqref{eq:circletrajectory}. In particular, $x\mapsto  \langle Px,x\rangle$ is a valid Lyapunov function, and the constant $c>0$ in the definition of $\rho$-hyperstability can be taken to be the condition number of $P^{1/2}$.

\subsection{Gradient descent} \label{subsect:GD}

As an illustration of Lemma \ref{lemm:circle}, consider the convergence of gradient descent for $f \in \mathcal{S}(m,L)$, where $0 < m < L$ (see \S \ref{subsect:main} for the definitions). By a coordinate shift, we can assume that $x_\star = 0$, namely $\nabla f(0) = 0$. The gradient descent iterates 
\[
x_{t+1}= x_t - \alpha \nabla f(x_t)
\]
on $\RR^n$ are of the form \eqref{eq:dynamicalsystem} with $A =C = I_n$ and $B = - \alpha I_n$. 

Let $\rho \in (0,1)$ and $k \in [m,L]$. In order for gradient descent to converge exponentially with rate $\rho$ when applied to quadratics $x\mapsto (k/2)|x|^2$, it is necessary and sufficient for the step size $\alpha > 0$ to be such that
\begin{equation} \label{eq:gradientlinearconstraint}
\rho \geq \rho_\GD \coloneqq \max(1-\alpha m, \alpha L -1).
\end{equation}
 According to the discussion in \S \ref{subsect:sector}, to show that the gradient descent iterates satisfy \eqref{eq:gradientlinearconvergence} for a given $\rho \geq \rho_\GD$, it suffices to show that $(I_n, -\alpha I_n; \sigma_0)$ is $\rho$-hyperstable (here $\sigma_0$ is given by \eqref{eq:sectorform}). This is an application of the circle criterion, Lemma \ref{lemm:circle}:

\begin{lemm} \label{lemm:GDhyperstable}
If $\rho \geq \rho_\GD$, then $(I_n,-\alpha I_n;\sigma_0)$ is $\rho$-hyperstable
\end{lemm} 
\begin{proof}
The transfer function as in Lemma \ref{lemm:circle} is $H(z) = -\alpha (z-1)^{-1}I_n$. After multiplying by $|z-1|^2$, the frequency condition of Lemma \ref{lemm:circle}  is equivalent to 
\[
\Re  \left[ (z-1+L\alpha)(\conj z-1+m\alpha)\right] \geq 0, \quad |z| = \rho.
\]
This is a linear function of $\Re z \in [-\rho,\rho]$, so the inequality above holds if and only if it holds for $\Re z = \pm \rho$. In other words, we must have
\begin{equation} \label{eq:gradientonepointFDI}
(\rho-1+L\alpha)( \rho -1 + m\alpha) \geq 0, \quad (\rho + 1-L\alpha)(\rho +1 - m\alpha) \geq 0.
\end{equation}
The first inequality holds if and only if $\rho \geq 1-\alpha m $, whereas the second holds if and only if $\rho \geq \alpha L -1$. In particular, any $\rho \geq \rho_{\GD}$ is feasible for \eqref{eq:FDI}. Furthermore, 
\[
A + kBC = (1 - \alpha k)I_n,
\]
so the eigenvalues of $A+kBC$ are certainly contained in the open disk of radius $\rho$ when $\rho \geq \rho_\GD$ and $k\in (m,L)$. Thus $(I_n,-\alpha I_n;\sigma_0)$ is $\rho$-minimally stable. Finally, Lemma \ref{lemm:circle} implies that  $(I_n,-\alpha I_n;\sigma_0)$  is $\rho$-hyperstable.
\end{proof}

In this case $(A,B) = (I_n, -\alpha I_n)$ is clearly controllable, so  \eqref{eq:LMI} with $\rho = \rho_\GD$ admits a solution $P \succeq 0$, which from the underlying block structure can assumed to be a multiple of the identity --- cf. the remarks in \cite[\S 4.2]{lessard2016analysis}. Furthermore, $P \succ 0$ by Lemma \ref{lemm:observable}.\footnote{In reference to the remarks preceding Lemma \ref{lemm:observable}, note that $\Pi(\bar z,z) \equiv 0$ when $|z| = \rho_\GD$ and $\alpha = 2/(m+L)$.}  Since $P$ has condition number equal to one, the constant $c>0$ in the definition of hyperstability can be chosen as $c=1$. Gradient descent is sufficiently simple that $P$ can also be found directly; this is precisely what is done in \cite{lessard2016analysis}, and we review the procedure to compare with Lemma \ref{lemm:GDhyperstable}.

Rather than parameterizing $P = pI_n$ for some $p > 0$, it is convenient to fix $p=1$ and instead introduce a multiplier $\lambda \geq 0$ for the constraint matrix associated to $\sigma_0$. Thus the aim is to minimize $\rho \geq \rho_\GD$ subject to feasibility of the LMI
\begin{equation} \label{eq:GDLMI}
\begin{bmatrix}
1 - \rho^2 & -\alpha \\
-\alpha & \alpha^2 
\end{bmatrix} + \lambda \begin{bmatrix}
-2Lm & L+m \\
L+m & -2
\end{bmatrix} \preceq 0.
\end{equation}
From the lower right entry, $\lambda$ must satisfy $\lambda \geq \alpha^2/2$. If $\alpha = 2/(L+m)$, then $\rho_\GD = (\kappa -1)/(\kappa+1)$, and \eqref{eq:GDLMI} is feasible for this value of $\rho$ by taking 
\[
\lambda = \alpha^2/2.
\]
If $\alpha \neq 2/(L+m)$, then feasibility for a given $\rho > 0$ requires that $\lambda > \alpha^2/2$. In that case, by Schur complements, feasibility holds if and only if $\rho$ satisfies
\[
\rho^2 \geq 1- 2Lm \lambda + \frac{((L+m)\lambda -\alpha)^2}{2\lambda -\alpha^2}.
\]
If $\alpha \neq 2/(L+m)$, then right hand side is minimized when 
\[
\lambda = \begin{cases} \alpha(1-\alpha m)/(L-m) &\text{ if } \alpha < 2/(L+m), \\
\alpha(\alpha L - 1)/(L-m) &\text{ if }  \alpha > 2/(L+m). \end{cases}
\]
This indeed yields $\rho \geq \rho_\GD$. Compare this to Lemma \ref{lemm:GDhyperstable}: the frequency-domain test requires minimizing $\rho$ subject to the trivial inequalities \eqref{eq:gradientonepointFDI}, which does not involve the decision variable $p$ (or equivalently $\lambda$).

\subsection{Jury--Lee criterion} \label{subsect:ZF}
As in \S \ref{subsect:sector}, consider an LTI system of the form \eqref{eq:dynamicalsystem}, and let $C \in \RR^{n\times d}$. As usual, let $0< m < L$.
 Although the sector IQC is certainly satisfied by gradients of functions in $\mathcal{F}_0(m,L)$, it does not exploit the fact that \eqref{eq:monotone} holds for arbitrary pairs $(y_1,y_2)$. More precisely, the sector IQC describes a certain property satisfied by the current iterate, but ignores that the same relationship holds between all the previous iterates as well.
  
  To capture memory of this additional structure for the most recent previous iterate, augment $x \in \ell_d$ with a lag variable $v \in \ell_n$ evolving according to
\[
v_{t+1} = LCx_t - u_t.
\]
The augmented state $(x,v)$ also satisfies an LTI system of the form \eqref{eq:dynamicalsystem} with system matrices $(\hat A, \hat B)$ given by
\begin{equation} \label{eq:hatAB}
\hat A = \begin{bmatrix}
A & 0 \\
LC & 0
\end{bmatrix}, \quad \hat B= \begin{bmatrix}
B \\ -I
\end{bmatrix}.
\end{equation}
With $y = Cx$ and $u = \nabla f (y)$, we then have $v_t = Ly_{t-1} - \nabla f(y_{t-1})$, which allows us to encode the constraint \eqref{eq:monotone} for the pair $(y_t, y_{t-1})$. 
More precisely, given $\tau \geq 0$, consider the IQC defined by the quadratic form
\begin{equation} \label{eq:augmentedform}
	\begin{aligned} 
\sigma_\tau(x,v,u) &= \sigma_\tau(x,v,u ; m,L,C) \\ &\coloneqq 2\langle u - mCx, LCx - u\rangle - 2\tau^2 \langle u-m Cx, v\rangle.
\end{aligned} 
\end{equation}
When $\tau = 0$ the form $\sigma_\tau$ is independent of $v$, and agrees with \eqref{eq:sectorform}. In \cite{lessard2016analysis}, the IQC defined by $\sigma_\tau$ is termed the off-by-one IQC. It is part of a larger class of so-called Zames--Falb IQCs satisfied by gradients of strongly convex functions \cite{boczar2015exponential,fetzer2017absolute}, but we will not have opportunity to use other IQCs from this family.

If $0 \leq \tau \leq \rho \leq 1$, then gradients of strongly convex function satisfy the $\rho$-IQC defined by $\sigma_\tau$ in the following precise sense. 

\begin{lemm} [{\cite[Lemma 10]{lessard2016analysis}}] \label{lemm:offbyone}
	 Let $L>m>0$, and suppose that $f \in \mathcal{F}_0(m,L)$. Given a real valued $x \in \ell_d$, define $v \in \ell_n$ by
	\begin{equation} \label{eq:vdef}
	v_t = LCx_t - \nabla f(Cx_t), \quad v_0 = 0. 
	\end{equation}
	If $\rho >0$ and $0 \leq \tau \leq \rho \leq 1$, then $(x,v,\nabla f(Cx)) \in \IQC(\sigma_\tau; \rho)$.
\end{lemm}

Suppose that $(\hat A, \hat B; \sigma_\tau)$ is $\rho$-hyperstable for a given $\rho \in (0,1]$ and $\tau \in [0,\rho]$. Under the hypotheses of Lemma \ref{lemm:offbyone}, there exists $c>0$ such that
\[
\| x_t \| \leq c \rho^k\|x_0\|, \quad t \geq 0
\]
whenever $x_{t+1} = Ax_t + B\nabla f(Cx_t)$ and $f \in \mathcal{F}_0(m,L)$. Note that $v_0 = 0$, so this term does not appear on the right-hand side.

\begin{rema} For a fixed $\rho \in (0,1]$, feasibility of \eqref{eq:LMI} with $\tau \in [0,\rho]$ as an additional decision variable is a convex program. This is because
	\[
	\sigma_\tau = \frac{\lambda \sigma_0 + \sigma_{\rho}}{1+\lambda}, \quad \lambda = (\rho/\tau)^2 -1 ,
	\]
	so it suffices to consider conic combinations of the fixed IQC $\sigma_\rho$ and the sector $\IQC$ $\sigma_0$ with multipliers as additional linear decision variables. In particular, it suffices to prove Lemma \ref{lemm:offbyone} for the case $\tau = \rho$.

\end{rema}

Next, we address the analogue of Lemma \ref{lemm:circle}. To compute the Popov function corresponding to $(\hat A, \hat B; \sigma_\tau)$, observe that with respect to the splitting $(x,v)$ of the state,
\[
(zI -\hat A)^{-1} = \begin{bmatrix}
	(zI-A)^{-1} & 0 \\
	z^{-1} LC(z-A)^{-1} & z^{-1}I
\end{bmatrix}, \quad (zI-\hat A)^{-1}\hat B = \begin{bmatrix}
(zI-A)^{-1}B \\ z^{-1}(LH(z) - I)
\end{bmatrix}, 
\]
where 
$H(z) = C(zI-A)^{-1}B$. A straightforward computation shows that the Popov function satisfies 
\[
\Pi(\bar z, z) = -2\Re \left[ (1-\tau^2/z)(I - mH(z))^* (I-LH(z)) \right].
\]

\begin{lemm} \label{lemm:jurylee}
	Let $\rho > 0$ and $\tau \geq 0$. If $B \neq 0$ and there exists $\varepsilon \in [0,1]$ such that
	\begin{equation} \label{eq:minimalaugmented}
	\begin{bmatrix}
	A + \left( (1-\varepsilon)m + \varepsilon L\right)BC & -\varepsilon \tau^2 B \\
	(1-\varepsilon)(L-m)C & \varepsilon \tau^2 I  
	\end{bmatrix}
	\end{equation}
	is $\rho$-Schur,
 then $(\hat A; \hat B;\sigma_\tau)$ is $\rho$-hyperstable if and only if 
	\begin{equation} \label{eq:jurylee}
	\Re \left[ (1-\tau^2 /z)(I - mH(z))^*(I-LH(z)) \right]\succeq 0 
	\end{equation}
	whenever $|z| = \rho$ and $\det(A-zI) \neq 0$.
\end{lemm}
\begin{proof} 
	As in Lemma \ref{lemm:circle}, the control $u = mCx$ makes the left-hand side of \eqref{eq:minimalQSR} vanish. The $\rho$-Schur property of \eqref{eq:minimalaugmented} verifies the remaining hypotheses of Lemma \ref{lemm:minimalstability}. Again, since $R = -2I$ has nonzero determinant, Theorem \ref{theo:popov} applies.
\end{proof}

  When $\rho =1$, the frequency condition \eqref{eq:jurylee} was first obtained by Jury--Lee \cite{jury1966stability}. Its relationship with the multiplier  approach of Zames--Falb was recognized immediately \cite{o1967frequency}. 
  
  One can also formally derive \eqref{eq:jurylee} from the KYP lemma via a Lur'e--Postnikov-type Lyapunov function of the form 
\begin{equation} \label{eq:lyapunovequivalent}
V(x) = \langle P \tilde x, \tilde x\rangle 
+ f(Cx) - f(0) - (m/2)\|Cx\|^2, \quad \tilde x = (x,\nabla f(Cx)),
\end{equation}
where $f \in \mathcal{F}_0(m,L)$ \cite{ahmad2014less}. In general it is necessary to consider the quadratic part as a function of the augmented state $\tilde x$ rather than the original state $x$. 

 Several recent works have considered Lyapunov functions for the analysis of optimization methods in the form \eqref{eq:lyapunovequivalent} (e.g., \cite{taylor2018lyapunov,taylor2019stochastic}), which in an appropriate sense is equivalent to the off-by-one IQC: as shown in \cite{ahmad2014less}, from the exponential decrease condition $V(x_{t+1}) \leq \rho V(x_t)$ one can derive an inequality of the form \eqref{eq:LMI} whose frequency-domain dual is equivalent to the Jury--Lee inequality \eqref{eq:jurylee}.

A more restrictive class of Lyapunov functions is
\begin{equation} \label{eq:weaklyapunov}
V(x) = \langle Px,x\rangle + f(Cx) - f(0) - (m/2)\| Cx \|^2.
\end{equation}
It is shown in \cite{hu2017dissipativity} that such a Lyapunov function suffices to recover standard convergence results for Nesterov's accelerated method. Using the constraints derived  in \cite[Lemma 4.1]{fazlyab2018analysis} or \cite{hu2017dissipativity}, one formally recovers from \eqref{eq:weaklyapunov} an earlier frequency condition due to Jury--Lee \cite{jury1964stability}.

\subsection{Nesterov's accelerated method} \label{subsect:nesterov}

Nesterov's accelerated gradient method for strongly convex functions on $\RR^n$ is of the form
\begin{equation} \label{eq:nesterov} 
	\begin{aligned}
		\xi_{k+1} &= \xi_k + \beta(\xi_k - \xi_{k-1}) - \alpha \nabla f(y_k), \\
		y_k &= \xi_k + \beta(\xi_k - \xi_{k-1}),
	\end{aligned}
\end{equation} 
where $\xi_k \in \RR^{n}$ (cf. the more general recursion \eqref{eq:degree2}). This is a feedback system of the form \eqref{eq:lure} with state $(\xi_k, \xi_{k-1})$ and system matrices
\[
A = \begin{bmatrix}
	1 + \beta & -\beta \\
	1 & 0
\end{bmatrix} \otimes I_n, \quad B = \begin{bmatrix}
	-\alpha \\ 0  
\end{bmatrix} \otimes I_n, \quad C = \begin{bmatrix}
	1+\beta& -\beta
\end{bmatrix} \otimes I_n.
\]
For the remainder of this section fix $0 < m < L$, and specify the parameters
\begin{equation} \label{eq:standardtuning}
	\alpha \coloneqq \frac{1}{L}, \quad \beta \coloneqq \frac{\sqrt{\kappa}-1}{\sqrt{\kappa}+1},
\end{equation}
where recall $\kappa = L/m$.
We refer to these choices as the \emph{standard tuning} for Nesterov's method. Also define
\[
\bar \rho_\AG = \bar \rho_\AG(m, L) = \sqrt {1 - \kappa^{-1/2}}.
\]
Nesterov's classical result \cite[\S 2.2]{nesterov2018lectures} states that for each $f \in \mathcal{F}(m,L)$ the iterates \eqref{eq:nesterov} (with the standard tuning) satisfy 
\[
\| \xi_{t} - \xi_\star \| \leq c \bar \rho_\AG^{\, t}\| \xi_0 - \xi_\star\|, \quad t \geq 0.
\]
Here $c>0$ is independent of $f$, and $\xi_\star$ is the global optimizer of $f$. 

By numerically solving the feasibility problem for \eqref{eq:LMI} using the off-by-one IQC, it was observed in \cite{lessard2016analysis} that the rate $\bar \rho_\AG$ is conservative. For a graphical illustration, see Figure \ref{fig:nesterov}. In Proposition \ref{prop:nesterovstandard} below, we give an analytic formula for the best rate that can be certified using the Jury--Lee criterion. For this we take $\tau = \rho$, but a more involved analysis shows that this rate is actually optimal over all admissible $\tau$; the details are omitted, although it can already be seen numerically in Figure \ref{fig:nesterov}.
\begin{prop} \label{prop:nesterovstandard}
	Let $0 < m < L$, and define $(\alpha,\beta)$ by \eqref{eq:standardtuning}. Denote by $r_- > 0$ the smallest positive root of
	\begin{equation} \label{eq:nesterovdiscriminant}
		\begin{aligned}
		h(r) &= 8(1-\sqrt{\kappa})^4 - (1-\sqrt{\kappa})^2(15\kappa -10\sqrt{\kappa} -1)r \\ &+ 2\kappa (3\kappa -4\sqrt{\kappa} + 1)r^2 + \kappa^2r^3,
	\end{aligned}
	\end{equation}
and set
\[
\rho_\AG = \rho_\AG(m,L) = r_-^{1/2}.
\]
If $\rho \geq \rho_\AG$, then there exists $c>0$ such that for any $f \in \mathcal{F}(m,L)$ the iterates \eqref{eq:nesterov} satisfy
	\[
	\| \xi_t - \xi_\star \| \leq c\rho^t \|\xi_0 - \xi_\star\|, \quad t \geq 0.
	\]
Furthermore, $\rho_\AG < \bar \rho_\AG$.
\end{prop}

The discriminant of $h(r)$ is positive, and it is easy to check that $h(0) > 0$ and $h'(0) < 0$. Consequently, $h(r)$ has one negative root and two positive roots $0 < r_- < r_+$. Standard asymptotic analysis reveals that
\[
r_- \sim 1 - \tfrac{4}{3}\kappa^{-1/2}, \quad \kappa \rightarrow \infty.
\]
In particular, for large $\kappa$, Nesterov's method with its standard tuning converges at a rate
\[
\rho_\AG \sim 1- \tfrac{2}{3}\kappa^{-1/2}, \quad \kappa \rightarrow \infty.
\]
Meanwhile, the more conservative classical rate has the asymptotics $\bar \rho_\AG \sim 1-\tfrac{1}{2}\kappa^{-1/2}$.

In order to prove Proposition \ref{prop:nesterovstandard}, we begin by examining the frequency condition \eqref{eq:jurylee}. Because of its block structure, for the purposes of verifying \eqref{eq:FDI} we can assume that the transfer matrix $H(z) = C(zI-A)^{-1}B$ is scalar, given by
\[
H(z) = -\alpha \frac{(1+\beta)z-\beta}{(z-1)(z-\beta)}.
\]
If $|z| = \rho$ for a given $\rho > 0$ and we fix $\alpha = 1/L$, then
\[
|z-1|^2|z-\beta|^2\Re \left [(1-\rho^2/z)(m H(z)-I)^*(LH(z)-I) \right] \\ 
= F(\Re z, \rho),
\]
where
\begin{align*}
	F(t,\rho) &\coloneqq \rho^2 \beta(1-\kappa)  + \rho^4((2+\beta)\kappa - 1 -\beta) \\ &+ (\rho^2(1+2\beta)(1-\kappa) - \rho^4\kappa)t + 2\beta(\kappa-1)t^2.
\end{align*}
If $\beta \geq 0$, then $F(t,\rho)$ is a convex quadratic function of $t$, and $\partial_t F(0,\rho) < 0$ for any $\rho$ since $\kappa > 1$. 

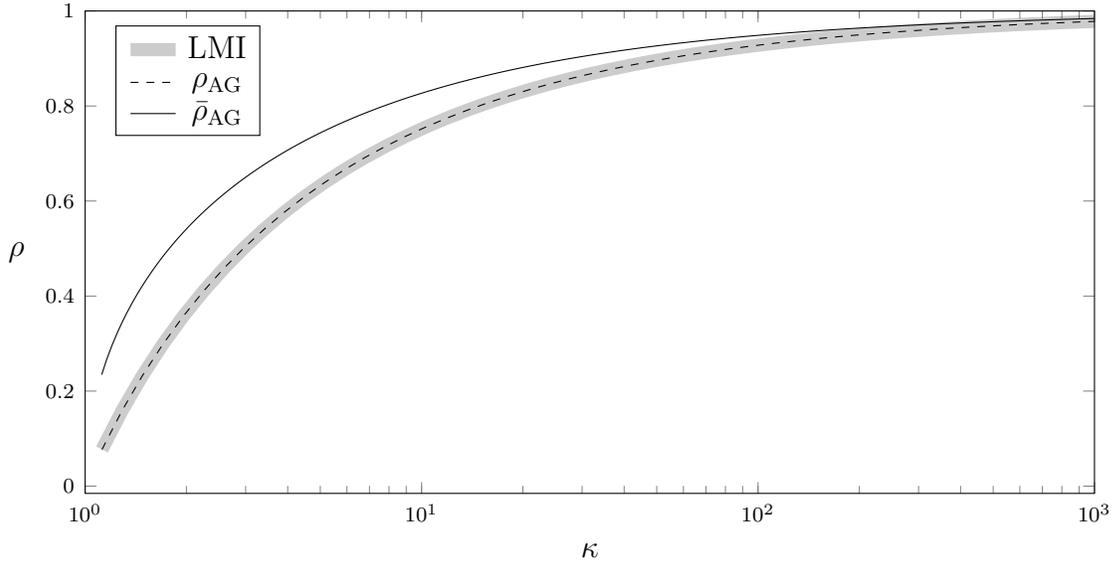
\begin{figure}
	\begin{center}
		\begin{tikzpicture}[
			 declare function ={ barrhoag(\k) = sqrt(1-1/sqrt(\k));
			},]
			\pgfkeys{/pgf/number format/1000 sep={}}
			\begin{axis}[
				height=8cm, width=15cm, xlabel=$\kappa$, ylabel=${\rho}$, ticklabel style = {font=\tiny}, legend pos = north west,
				ylabel near ticks, 
				axis line style={-},
				xmode = log,
				xmin = 1,
				xmax = 1000,
				ymax = 1,
				xlabel style={yshift=0cm},
				ylabel style={rotate=-90},
				]
				\addplot[color=black, opacity=.2, line width=5pt] table[col sep=comma,header=false,x index=0,y index=1] {fgm_off_by_one.csv};					
				\addplot[color=black, dashed] table[col sep=comma,header=false,x index=0,y index=1] {fgm_cubic_root.csv};
				\addplot[color=black, domain = 1.12 : 1000, samples=200] {barrhoag(x)};
				\legend{LMI, $\rho_\AG$, $\bar\rho_\AG$}
			\end{axis}
		\end{tikzpicture}
	\end{center}
	\caption{A comparison of the improved rate $\rho_\AG$ from Proposition \ref{prop:nesterovstandard} versus the classical rate $\bar \rho_\AG$ as a function of $\kappa$. The thick gray line is the numerically computed optimal $\rho$ such that \eqref{eq:LMI} is feasible for $P \succeq 0$ using the off-by-one IQC (over all admissible $\tau$). \label{fig:nesterov}}
\end{figure}

Now fix $\beta$ according to \eqref{eq:standardtuning}. The remarks above show that $F(t,\rho)$ is nonnegative for $t \in [-\rho,\rho]$ if and only if it is nonnegative for $t \in [0,\rho]$.  Define
\[
\rho_\star \coloneqq \inf \{\rho > 0: F(t,\rho) \geq 0 \text{ for all $t \in [0,\rho]$} \}.
\]
This infimum is finite, since it can be checked that $\rho =1$ is always feasible.  In fact, $\rho_\star < 1$ since the derivative of $\rho \mapsto F(\rho, \rho)$ at $\rho = 1$ is $-1$ and $\partial_t F(t,1)$ is strictly negative for $t\in [0,1]$. By continuity,
 \[
F(t, \rho_\star) \geq 0,\quad t \in [0,\rho_\star].
\]
It is easy to see that $\rho_\star > 0$ as expected, since $F(0,\rho) < 0$ for arbitrarily small positive $\rho$.  Note that any root of $F(t,\rho_\star)$ must be nonnegative, since $F(0,\rho_\star) \geq 0$ and $\partial_t F(0,\rho_\star) < 0$.

In order to compute $\rho_\star$, we argue that
 \[
\mathrm{disc}(F(t,\rho_\star);t) = 0.
 \]
Consider the alternatives: \begin{inparaenum}
 	\item If the discriminant is negative, then $F(t,\rho_\star) > 0$ for all $t \in [0,\rho_\star]$. The latter is an open condition with respect to $\rho$, which contradicts the definition of $\rho_\star$ as an infimum.
 	
 	\item If the discriminant is positive, then we must have $\partial_t F(\rho_\star, \rho_\star) \leq 0$, otherwise $F(t,\rho_\star)$ would be negative somewhere in the interval $[0, \rho_\star]$. Furthermore, 
 	\[
 	F(\rho_\star, \rho_\star) =0,
 	\] 
 	since otherwise we would again have $F(t, \rho_\star) > 0$ for $t \in [0,\rho_\star]$. Now 
 	\[
 	F(\rho, \rho) = \rho^2(1-\rho)((1-\rho)\sqrt{\kappa} -1)^2,
 	\]
 	which would imply $\rho_\star = 1 - 1/\sqrt{\kappa}$. In that case, however, a direct calculation shows that 
 	\[
 	\partial_t F(\rho_\star, \rho_\star) = 2\kappa^{-1}(\sqrt{\kappa} - 1)^3 > 0,
 	\]
 	which is a contradiction. 
 \end{inparaenum}

\begin{proof} [Proof of Proposition \ref{prop:nesterovstandard}]
The discriminant of $F(t,\rho)$ with respect to $t$ is $\rho^2 h(\rho^2)$, where $h(r)$ is the cubic polynomial \eqref{eq:nesterovdiscriminant}. Therefore $\rho_\AG \coloneqq \rho_\star$ is the square root of the smallest positive root of $h(r)$. To apply Lemma \ref{lemm:jurylee}, we verify that \eqref{eq:minimalaugmented} is $\rho_\AG$-Schur when $\varepsilon = 0$. Owing to the block structure, this is the same as showing that the matrix
\[
 \begin{bmatrix}
2(1- 1/\sqrt{\kappa}) & -(1-1/\sqrt{\kappa})^2 & 0 \\
1 & 0 & 0 \\
2 L(1-1/\sqrt{\kappa}) &-L(1-1/\sqrt{\kappa})^2 & 0
\end{bmatrix} 
\] 
is $\rho_\AG$-Schur. The eigenvalues of this matrix are $0$ and $1-1/\sqrt{\kappa}$. 
Now compute
\[
h((1-1/\sqrt{\kappa})^2) = \frac{4(\sqrt{\kappa}-1)^6}{\kappa^2} >0, \quad h'((1-1/\sqrt{\kappa})^2) = \frac{-12(\sqrt{\kappa} -1)^5}{\kappa} < 0.
\]
The cubic nature of $h$ implies that $1-1/\sqrt{\kappa} < \rho_\AG$, so the matrix above is indeed $\rho_\AG$-Schur. 

To finish the proof of Proposition \ref{prop:nesterovstandard}, it remains to show that $\bar \rho_\AG > \rho_\AG$. For this, compute
\[
h(\bar \rho_\AG^2) = \frac{(\sqrt{\kappa}-1)^4}{\kappa} >0, \quad  h'(\bar \rho_\AG^2) = \frac{2(\sqrt{\kappa} -1)^3(3\sqrt{\kappa}+1)}{\sqrt{\kappa}} > 0.
\]
Again, the cubic nature of $h(r)$ implies that $\bar \rho_\AG^2 > r_+$, and hence $\bar \rho_\AG > \rho_\AG$.
\end{proof}

\subsection{Inexact inputs} \label{subsect:noise}
We begin this section with some abstract considerations, and then discuss applications to feedback systems. Consider an LTI system of the form
\[
x_{t+1} = Ax_t + B_{ue}\begin{bmatrix}
	u_t \\ e_t
\end{bmatrix}, \quad B_{ue} = \begin{bmatrix}
B_u & B_e
\end{bmatrix},
\]
where $x \in \ell_d$ and $u,e \in \ell_n$. We assume that $(u,e)$ satisfy the IQC $\sigma_\Delta(u_t, e_t) \geq 0$ for all $t \geq 0$, where
\[
\sigma_\Delta(u, e) = \delta^2 \|u\|^2 - \|e \|^2
\]
for some fixed $\delta \in [0,1)$. We also fix some reference quadratic form 
\[
\sigma_\refe(x,u) =  \langle Qx,x\rangle + 2 \Re\langle Sx,u \rangle + \langle Ru,u\rangle,
\]
which we insist is independent of $e$. Our goal is to characterize the $\rho$-hyperstability of $(A, B_{ue}; \sigma_\refe, \sigma_\Delta)$. As discussed in \S \ref{subsect:feedback}, it suffices to show that $(A,B_{ue}; \sigma_\refe + \lambda \sigma_\Delta)$ is $\rho$-hyperstable for some $\lambda \geq 0$.

 Note that $\rho$-hyperstability of $(A, B_{ue};\sigma_\refe+ \lambda \sigma_\Delta)$ implies the same for $(A,B_u; \sigma_{\mathrm{ref}})$. It is also easy to see that $\rho$-minimal stability of $(A,B_u; \sigma_{\mathrm{ref}})$ implies the same for $(A,B_{ue}; \sigma_{\mathrm{ref}}+ \lambda \sigma_\Delta)$:

\begin{lemm} \label{lemm:perturbedminimallystable}
	Let $\rho > 0$. If $(A,B_u; \sigma_{\mathrm{ref}})$ is $\rho$-minimally stable, then $(A, B_{ue} ;\sigma_\refe+ \lambda \sigma_\Delta)$ is $\rho$-minimally stable for any $\lambda \geq 0$.
\end{lemm}
\begin{proof}
	Given $x^0 \in \CC^m$, let $(x,u) \in \sol(A,B_u) \cap \IQC(\sigma_\refe; \rho)$ be as in the definition of $\rho$-minimal stability. Obviously 
	\[ 
	(x,u,0) \in \sol(A, B_{ue}) \cap \IQC(\sigma_\refe+ \lambda \sigma_\Delta; \rho),
	\]
	since adding $\lambda \delta^2 |u_t|^2$ to the left-hand side of \eqref{eq:IQC} has a favorable sign.
\end{proof} 

Similarly, if $(A,B_u)$ is controllable, then so is $(A, B_{ue})$. Finally, observe that according to Lemma \ref{lemm:Rconstraint}, a necessary condition for hyperstability of $(A, B_{ue};\sigma_\refe+ \lambda \sigma_\Delta)$ is 
\begin{equation} \label{eq:Rlambdaconstraint}
	\lambda \delta^2 I \preceq -R.
\end{equation}
This provides an a priori bound on $\lambda$ that will prove useful later on.

Next, we record the form taken by \eqref{eq:FDI} for the triple $(A,B_{ue}, \sigma_\refe + \lambda \sigma_\Delta)$. Let $\Pi_{ue}(\zeta,z)$ denote the Popov function corresponding to this triple. Define
\[
G_u(z) \coloneqq (zI-A)^{-1}B_u, \quad G_e(z) \coloneqq (zI -A)^{-1}(B_u - B_e).
\]
The reason for considering the difference $B_u - B_e$ is that in our applications this term either vanishes or is otherwise easy to handle. Denote by 
\[
\Pi_u(\bar z,z) = G_u(z)^*QG_u(z) + 2\Re SG_u(z) + R
\] 
 the Popov function for the triple $(A,B_u;\sigma_\refe)$. Similarly, let $\Pi_e(\bar z,z)$ denote the Popov function for $(A,B_u - B_e; \sigma_\refe )$. Finally, define the off-diagonal part
\[
\Psi(z) = G_e(z)^*QG_u(z) + SG_u(z) + (SG_e(z))^* + R.
\]
The Popov function corresponding to $(A, B_{ue};\sigma_\refe+ \lambda \sigma_\Delta)$ satisfies
\begin{equation}  \label{eq:noisyblockFDI}
	\begin{bmatrix}
		I & I \\
		0 & -I
	\end{bmatrix}^*
	\Pi_{ue}(\bar z, z)
	\begin{bmatrix}
		I & I \\
		0 & -I
	\end{bmatrix} = \begin{bmatrix}
		\Pi_u(\bar z, z) + \lambda \delta^2 I &\Psi(z)^* + \lambda \delta^2 I\\
		\Psi(z)+ \lambda \delta^2 I& \Pi_e(\bar z,z) - \lambda (1-\delta^2)I
	\end{bmatrix}. 
\end{equation}
Of course this conjugation preserves definiteness properties of $\Pi_{ue}(\bar z,z)$, but it leads to more transparent computations.

Now we consider the two cases of interest. First is the case of a perturbed feedback system 
\begin{equation} \label{eq:perturbedfeedback}
x_{t+1} = Ax_t + B( \nabla f(Cx_t) + e_t),
\end{equation}
where $f \in \mathcal{S}_0(m,L)$. In this case we have
\begin{equation} \label{eq:Buesector}
B_u = B_e = B.
\end{equation}
For the reference IQC we take $\sigma_\mathrm{ref} = \sigma_0$, where the sector IQC $\sigma_0$ is given by \eqref{eq:sectorform} and only acts on the $(x,u)$ variables. In particular $G_e(z) \equiv 0$, so
\[
 \Pi_e(\bar z,z) - \lambda (1-\delta^2)I = -(2+ \lambda(1-\delta^2))I.
\]
Clearly this term is negative definite for $\lambda \geq 0$; taking Schur complements in \eqref{eq:noisyblockFDI} and using that this term is a multiple of the identity, we deduce the following:
\begin{lemm} \label{lemm:perturbedFDIsector}
Fix $(A,B)$ and consider the triple $(A, B_{ue}; \sigma_0 + \lambda \sigma_\Delta)$, where $B_{ue}$ is determined by \eqref{eq:Buesector}. If $\rho > 0$, then \eqref{eq:FDI} holds if and only if
\begin{multline} \label{eq:perturbedFDIsector}
2\lambda \Re \left[ (I-mH(z))^*(I-LH(z)) \right] \\ - ((L-m)^2 +2 Lm \lambda \delta^2) H(z)^*H(z) - \lambda^2\delta^2 I \succeq 0
\end{multline}
whenever $|z| = \rho$ and $\det(A-zI) \neq 0$, where $H(z) = C(zI-A)^{-1}B$. 
\end{lemm}
This is an inexact version of the circle criterion.

\begin{rema} \label{rema:lambda=0} Lemma \ref{lemm:perturbedFDIsector}  also makes it clear that \eqref{eq:FDI} cannot hold for $\lambda = 0$ unless $H(z) \equiv 0$. 
\end{rema}

The second case also concerns the perturbed system \eqref{eq:perturbedfeedback}, but now we make the stronger assumption that $f \in \mathcal{F}_0(m,L)$. Here we take the reference IQC to be the off-by-one IQC from \S \ref{subsect:ZF}. This of course requires augmenting the original state space. Thus we consider the system
\begin{align*}
	\begin{bmatrix}
		x_{t+1} \\
		v_{t+1}
	\end{bmatrix}	= \hat A \begin{bmatrix}
		x_t \\ v_t
	\end{bmatrix} + \hat B u_t + \hat B e_t + \begin{bmatrix}
		0 \\ I
	\end{bmatrix} e_t,
\end{align*}
where $(\hat A, \hat B)$ are defined by \eqref{eq:hatAB}. The final counterterm involving $e_t$  ensures that $v_{t}$ evolves according to $v_{t+1} = LCx_t - u_t$, as needed to apply Lemma \ref{lemm:offbyone}. Thus we take
\begin{equation} \label{eq:Bueconvex}
B_u = \hat B, \quad B_e = \hat B + \begin{bmatrix}
	0 \\ I
\end{bmatrix}.
\end{equation}
Let  $\sigma_{\mathrm{ref}} = \sigma_\tau$,  where $\sigma_\tau$ is given by \eqref{eq:augmentedform}. This time, the bottom right entry of \eqref{eq:noisyblockFDI} is
\begin{equation*} \label{eq:bottomright}
 \Pi_e(\bar z,z) - \lambda (1-\delta^2)I  = (2\Re (\tau^2/z - 1) - \lambda (1-\delta^2))I.
\end{equation*}
Recall $\tau$ must be constrained by $\tau \in [0,\rho]$ in order to test for $\rho$-hyperstability. In particular, the bottom right entry of \eqref{eq:noisyblockFDI} is negative definite for $|z| = \rho$ and $\lambda \geq 0$ (except in the degenerate case $\tau = \rho=1$ and $\lambda = 0$, but this is irrelevant since negativity only fails at a single point $z = 1$). By taking Schur complements we obtain after some algebra the following:
\begin{lemm} \label{lemm:perturbedFDIconvex}
Fix $(A,B)$ and consider the triple $(\hat A, B_{ue}; \sigma_\tau + \lambda \sigma_\Delta)$, where $\hat A$ is determined by \eqref{eq:hatAB} and $B_{ue}$ is given by \eqref{eq:Bueconvex}. If $\rho > 0$ and $\tau \in [0,\rho]$, then \eqref{eq:FDI} holds if and only if
		\begin{multline} \label{eq:perturbedFDIconvex}
			2\lambda \Re \left[ (1-\tau^2/z)(I-mH(z))^*(I-LH(z)) \right]  \\- ((L-m)^2 |1-\tau^2/z|^2 + 2\lambda\delta^2 L m \Re[1-\tau^2/z]) H(z)^*H(z)-\lambda^2 \delta^2 I \succeq 0
		\end{multline}
whenever $|z| = \rho$ and $\det(A-zI) \neq 0$, where $H(z) = C(zI-A)^{-1}B$. 
\end{lemm}
This is an inexact version of the Jury--Lee criterion.
The same considerations as in Remark \ref{rema:lambda=0} show that \eqref{eq:FDI} cannot hold with $\lambda = 0$ unless $H(z) \equiv 0 $.

\section{Inexact gradient descent} \label{sect:inexactGD}

\subsection{Functions with sector bounded gradients} \label{subsect:GDSml}

This section extends the discussion  in \S \ref{subsect:GD} of gradient descent  over $\mathcal{S}(m,L)$ to the inexact case. Let $0 < m < L$ and $\delta \in (0,1)$.  Consider the iterates
\[
x_{t+1} = x_t - \alpha (\nabla f(x_t) + e_t)
\] 
on $\RR^n$, where $f \in \mathcal{S}(m,L)$ is normalized by $\nabla f(0) = 0$, and $e \in \ell_n$ is an arbitrary sequence satisfying $\| e_t \| \leq \delta \| \nabla f(x_t) \|$. Two valid choices for $e$ are
\[
e= \pm \delta \nabla f(x),
\]
 which correspond to gradient descent with step-sizes $(1\pm\delta)\alpha$. Any exponential convergence rate must therefore satisfy $\rho \geq \rho_\GD(\delta)$, where 
\begin{equation} \label{eq:gradientlinearconstraintnoise}
\rho_\GD(\delta) = \max(1-\alpha m (1-\delta), \alpha L (1+\delta) -1).
\end{equation}

We now apply the results of \S \ref{subsect:noise} to the  LTI system $x_{t+1} = x_t - \alpha (u_t + e_t)$. In the notation of \S \ref{subsect:noise}, the system matrices are $A = C = I_n$ and $B_u = B_e = -\alpha I_n$. In particular,
\[
B_{ue} = -\alpha \begin{bmatrix}
	1 & 1
\end{bmatrix} \otimes I_n
\] 
We choose $\sigma_{\mathrm{ref}} = \sigma_0$ as the reference IQC. The goal is to compute
\begin{equation} \label{eq:rhostar}
\rho_\star =  \inf\{ \rho \geq  \rho_\GD(\delta): (A,B_{ue}; \sigma_0 + \lambda \sigma_\Delta) \text{ is } \text{$\rho$-hyperstable} \text{ for some } \lambda \in (0,2\delta^{-2}]\}
\end{equation}
and show that this infimum is attained. The constraint $\lambda \in (0,2\delta^{-2}]$, which arises from \eqref{eq:Rlambdaconstraint} and Remark \ref{rema:lambda=0},  is a necessary condition for hyperstability.

\begin{lemm} 
If $\rho > 0$ and $\lambda \in (0,2\delta^{-2}]$, then $(A,B_{ue}; \sigma_0 + \lambda \sigma_\Delta)$ is $\rho$-hyperstable if and only \eqref{eq:FDI} holds.	
\end{lemm} 
\begin{proof} 	
Since $B \neq 0$, hyperstability automatically implies \eqref{eq:FDI}. The converse follows from Theorem \ref{theo:popov}: the minimal stability hypothesis is satisfied according to Lemma \ref{lemm:GDhyperstable} and Lemma \ref{lemm:perturbedminimallystable}. To verify the last hypothesis of Theorem \ref{theo:popov}, let $\Pi_{ue}(\zeta,z)$ denote the Popov function corresponding to $(A,B_{ue}; \sigma_0 + \lambda \sigma_\Delta)$. Then
\[
\Pi_{ue}(\infty,z) = \begin{bmatrix}
(L+m)H(z) +(\lambda\delta^2 -2)I & H(z) \\
0 & -\lambda I
\end{bmatrix},
\]
where $H(z) = -\alpha(z-1)^{-1}I$. Since $\lambda \neq 0$, the determinant of this rational matrix does not vanish identically.
\end{proof} 

Let $\lambda \in (0,2\delta^{-2}]$, and suppose that $z$ not an eigenvalue of $A$. We apply the inexact circle criterion, Lemma \ref{lemm:perturbedFDIsector}: if \eqref{eq:perturbedFDIsector} is multiplied through by $|z-1|^2$, then $\Pi_{ue}(\bar z,z) \preceq 0$ if and only if
\begin{multline} \label{eq:onepointnoiseFDI}
-\alpha^2 (L-m)^2+2mL \alpha^2\lambda(1-\delta^2) \\ +2\alpha\lambda (m+L)\Re (z-1) 
+\lambda(2 -\delta^2\lambda) |z-1|^2 \geq 0.
\end{multline}
When restricted to the circle of radius $\rho$, the left-hand side of \eqref{eq:onepointnoiseFDI} is a linear function of $\Re z \in [-\rho,\rho]$.  Therefore \eqref{eq:onepointnoiseFDI} holds for all $|z| = \rho$ if and only if it holds for $z= \pm \rho$. For fixed parameters $(L,m,\alpha,\delta)$, define
\begin{multline*}
F(t,\lambda) = -\alpha^2(L-m)^2+2mL \alpha^2 \lambda(1-\delta^2) \\ +2\alpha\lambda (m+L)(t-1) 
+ \lambda(2 -\delta^2\lambda)(t-1)^2.
\end{multline*}
Thus \eqref{eq:onepointnoiseFDI} holds for a given $\lambda \geq 0$ if and only if $F(\pm \rho,\lambda) \geq 0$ for both choices of sign. In particular, $\rho_\star$ can be computed as 
\[
\rho_\star = \inf\{\rho \geq \rho_{\GD}(\delta): F(\pm\rho,\lambda) \geq 0 \text{ for some } \lambda \in (0,2\delta^{-2}] \}.
\]
By definition, $\rho_\star \geq \rho_{\GD}(\delta)$. Since we do not impose any upper bound on $\rho$, the quadratic nature of $F$ in $t$ implies that the infimum is finite. Since $F$ is continuous and $\lambda$ is bounded, there exists $\lambda_\star \in [0,2\delta^{-2}]$ for which $F(\pm \rho_\star,\lambda_\star) \geq 0$. According to Remark \ref{rema:lambda=0}, we in fact have $\lambda_\star \in (0,2\delta^{-2}]$. Also,
\begin{equation} \label{eq:endpoints}
\begin{aligned}
F(1-\alpha m (1-\delta),\lambda) &= -\alpha^2 (L+m (\delta ^2 \lambda -\delta  \lambda -1 ) )^2,
 \\ F(1-\alpha L(1+\delta),\lambda) &= - \alpha^2 (m + L(\delta ^2 \lambda +\delta  \lambda -1 ) )^2
\end{aligned}
\end{equation} 
 for all $\lambda$. The following auxiliary lemma establishes some useful properties of $F(\pm \rho_\star, \lambda_\star)$.
 \begin{lemm} \label{lemm:inexactSproperties}
	The pair $(\rho_\star,\lambda_\star)$ satisfies the following properties.
	\begin{enumerate} \itemsep6pt
		\item $\min (F(\rho_\star,\lambda_\star), F(-\rho_\star,\lambda_\star))  = 0$.
		
		\item $\lambda_\star \in (0,2\delta^{-2})$.
		
		\item If $\max (F(\rho_\star,\lambda_\star), F(-\rho_\star,\lambda_\star)) > 0$, then $\rho_\star = \rho_\GD(\delta)$. 
	\end{enumerate}	
\end{lemm}
\begin{proof}
	\begin{inparaenum}
		\item If $F(\pm \rho_\star,\lambda_\star) >0$ for both choices of sign, then by continuity of $F$ and the definition of $\rho_\star$ we must have $\rho_\star = \rho_{\GD}(\delta)$. This contradicts \eqref{eq:endpoints}.
		
		\item  If $\lambda_\star = 2\delta^{-2}$, then $F(t,2\delta^{-2})$ reduces to a strictly increasing linear function of $t$. Since we know that $F(t,\lambda_\star)$ must vanish at one of the endpoints $\pm \rho_\star$ by the first part, it must do so at $-\rho_\star$, i.e.,
		\[
		F(-\rho_\star,2\delta^{-2}) = 0.
		\]
		On the other hand,  it follows from \eqref{eq:endpoints} that $F(1-\alpha L(1+\delta),2\delta^{-2}) <0$, so 
		\[
		-\rho_\star > 1 -\alpha L (1+\delta) \geq -\rho_{\GD}(\delta),
		\]
		 which is a contradiction.
		 
		\item Suppose that $F(\rho_\star,\lambda_\star) > 0$. By the first part we have $F(-\rho_\star,\lambda_\star) = 0$. If $\partial_\lambda F(-\rho_\star,\lambda_\star) \neq 0$ but $\rho_\star > \rho_\GD(\delta)$, then by perturbing $\rho$ and $\lambda$ slightly we can contradict the definition of $\rho_\star$. This is possible since $\lambda_\star$ is in the open interval $(0,2\delta^{-2})$. Thus we either have $\rho_\star = \rho_{\GD}(\delta)$ or $\partial_\lambda F(-\rho_\star,\lambda_\star) = 0$. But in the latter case, the simultaneous equations
		\[
		\partial_\lambda F(-\rho_\star,\lambda_\star) = 0, \quad F(-\rho_\star,\lambda_\star) = 0
		\]
		imply that $\rho_\star = \alpha L (1+\delta) -1$ or $\rho_\star =  \alpha m (1-\delta)-1$. Since $\rho_\star \geq \rho_\GD(\delta) \geq \alpha L(1+\delta)-1$, this implies $\rho_\star = \alpha L(1+\delta)-1 = \rho_\GD(\delta)$ as well. The same argument applies with roles of $\rho_\star$ and $-\rho_\star$ reversed.
	\end{inparaenum}
\end{proof}

It remains to compute $\rho_\star$. First we characterize the parameters $(m,L,\alpha,\delta)$ for which $\rho_\star = \rho_\GD(\delta)$. Recall that $\kappa = L/m$.

\begin{lemm} \label{lemm:achievedS}
	Let $0 < m< L$ and $\delta \in (0,1)$. If $\alpha > 0$, then $\rho_\star = \rho_\GD(\delta)$ if and only if one of the following conditions holds.
	\begin{enumerate} \itemsep6pt 
		\item $\delta < 2/(\kappa + 1)$ and 
		\[
		\alpha \leq \alpha_-(m,L,\delta) =  \frac{1}{1-\delta}\left( \frac{2}{L+m} - \frac{\delta}{m} \right). 
		\]
		\item $\delta \in [0,1)$ and
		\[
		\alpha \geq \alpha_+(m,L,\delta) = \frac{1}{1+\delta}\left( \frac{2}{L+m} + \frac{\delta}{L} \right).
		\]
	\end{enumerate}
\end{lemm}
\begin{proof}
We have $\rho_\GD(\delta) = 1-\alpha m(1-\delta)$ or $\rho_\GD(\delta) = \alpha L (1+\delta)-1$ depending on which is larger.

\begin{inparaenum} \item First suppose that $\rho_\GD(\delta) = 1-\alpha m(1-\delta)$. If $\rho_\star = \rho_\GD(\delta)$, then from \eqref{eq:endpoints} we must have $F(\rho_\star, \lambda_\star) = 0$, which uniquely determines
\begin{equation} \label{eq:lambdastar1}
\lambda_\star = \frac{\kappa-1}{\delta(1-\delta)}.
\end{equation}
 Now $F$ is quadratic in $t$, so $F(-\rho_\star,\lambda_\star) \geq 0$ if and only if $\partial_t F(0,\lambda_\star) \leq 0$, where
\[
\partial_t F(0,\lambda_\star) = 2\lambda_\star (\alpha  (m+L) - 2 + \delta^2 \lambda_\star ).
\]
This is true if and only if $\lambda_\star \leq (2 - \alpha (m+L))\delta^{-2}$, which combined with \eqref{eq:lambdastar1} is equivalent to $\alpha \leq \alpha_-$. Conversely, if $\alpha \leq \alpha_-$, then 
\[
1- \alpha m (1-\delta)> \alpha L(1+\delta)-1
\]
and hence $\rho_\GD(\delta) = 1- \alpha m(1-\delta)$. If $\lambda_\star$ is defined by \eqref{eq:lambdastar1}, then $F(\pm \rho_\GD(\delta), \lambda_\star) \geq 0$, which shows that $\rho_\star = \rho_\GD(\delta)$. Finally, observe that $\alpha_- > 0$ if and only if 
\[
\delta < 2/(\kappa+1),
\] 
which is also equivalent to $2 -\delta^2 \lambda_\star > 0$ (cf. the second part of Lemma \ref{lemm:inexactSproperties}).

\item A similar argument applies if $\rho_\GD(\delta) = \alpha L (1+\delta)-1$. If $\rho_\star = \rho_\GD(\delta)$, then 
\[
\lambda_\star = \frac{\kappa-1}{\kappa\delta(1+\delta)}.
\] 
in order to ensure $F(-\rho_\star,\lambda_\star) = 0$.  The condition $F(\rho_\star,\lambda_\star) \geq 0$, or equivalently $\partial_t F(0,\lambda_\star) \geq 0$, is the same as $\alpha \geq \alpha_+$. Conversely, if $\alpha \geq \alpha_+$, then $\rho_\GD(\delta) = \alpha L (1+\delta)-1$ and $F(\pm \rho_\GD(\delta), \lambda_\star) \geq 0$.
\end{inparaenum}
\end{proof}

In the context of Lemma \ref{lemm:achievedS}, we have $\rho_\GD(\delta) = \alpha L(1+\delta)-1$ when  $\alpha \geq \alpha_+$. In that case there is no restriction on $\delta$, but observe that $\delta < 2/(\kappa+1)$ is necessary for $\rho_\GD(\delta) < 1$ to hold; if $\alpha = \alpha_+$, then it is also sufficient.
\begin{proof} [Proof of Proposition \ref{prop:noisyS}]
The Proposition follows from Lemma \ref{lemm:achievedS} and the definition of $\rho_\star$. To show that one can take $c=1$ in the definition of hyperstability, note that $(A,B_{ue})$ is  controllable for $\alpha > 0$. Using Lemma \ref{lemm:KYP} and Lemma \ref{lemm:observable}, there exists $P = pI$ satisfying \eqref{eq:LMI} with $p>0$.
\end{proof} 

When the conditions of Lemma \ref{lemm:achievedS} are violated, we must have $\rho_\star > \rho_\GD(\delta)$. By the third part of Lemma \ref{lemm:inexactSproperties}, it is easy to compute $\rho_\star$. This leads to Proposition \ref{prop:weakGD}:

\begin{proof} [Proof of Proposition \ref{prop:weakGD}]
If $\rho_\star > \rho_\GD(\delta)$, then we must have $F(\rho_\star,\lambda_\star) = F(-\rho_\star,\lambda_\star) = 0$. Since $F$ is quadratic in $t$, this happens if and only if $\partial_t F(0,\lambda_\star) = 0$, or equivalently,
\[
\lambda_\star = \frac{2- \alpha(L+m)}{\delta ^2}.
\]
The second part of Lemma \ref{lemm:inexactSproperties} gives  an upper bound $\alpha < 2/(L+m)$. We can then compute
\[
\rho_\star = \left( 1-\frac{2 \alpha  L m}{L+m} + \frac{\alpha  \delta ^2 (L+ m -2 \alpha  L m )}{2- \alpha  (L+m)} \right)^{1/2}
\]
as claimed. Conversely,  the proof of Lemma \ref{lemm:achievedS} makes it clear that this value of $\rho_\star$ is strictly larger than $\rho_\GD(\delta)$ unless $\alpha = \alpha_-$ or $\alpha = \alpha_+$. The same argument as in the proof of Proposition \ref{prop:weakGD} shows that one can take $c = 1$ in the definition of hyperstability.
\end{proof}

\subsection{Strongly convex functions} \label{subsect:inexactFmL}

Next, we address Proposition \ref{prop:GDM} for strongly convex functions in $\mathcal{F}(m,L)$. As usual, we assume that $f \in \mathcal{F}_0(m,L)$ by translations. We again apply the results of \S \ref{subsect:noise}, using the off-by-one IQC as the reference IQC. The underlying LTI system is
\begin{align*}
	\begin{bmatrix}
		x_{t+1} \\
		v_{t+1}
	\end{bmatrix}	= \hat A \begin{bmatrix}
		x_t \\ v_t
	\end{bmatrix} + B_{ue} \begin{bmatrix}
	u_t \\ e_t
\end{bmatrix}, \quad \hat A = \begin{bmatrix}
1 & 0 \\
L & 0
\end{bmatrix} \otimes I_n, \quad B_{ue} = \begin{bmatrix}
-\alpha & -\alpha \\
-1 & 0
\end{bmatrix} \otimes I_n.
\end{align*}
First, we verify that the inexact Jury--Lee criterion \eqref{eq:perturbedFDIconvex} holds for some $\lambda > 0$ when
\[
\tau = \rho = \rho_\GD(\delta).
\]
Denote by $\Phi(z, \lambda)$ the left-hand side of \eqref{eq:perturbedFDIconvex} and consider the two possible values of $\rho_\GD(\delta)$.

\begin{inparaenum}
	\item First, suppose that $\rho_\star := \rho_\GD(\delta) = 1 - \alpha m (1-\delta)$, and calculate
	\[
	\Phi(\rho_\star,\lambda)  = -m^2 \alpha^2 (1-\delta)^2(\delta \lambda - (L-m)\alpha)^2.
	\]
	Requiring this quantity to be nonnegative uniquely determines
	\begin{equation} \label{eq:lambdastarconvexGD}
	\lambda_\star = \frac{(L-m)\alpha}{\delta}.
	\end{equation}
	Plugging in this value of $\lambda$, we can write $	\Phi(z, \lambda_\star) = F(\Re z)$, where
	\[
F(t) = 2\lambda_\star(z - \rho_\star)(2(1 - \alpha L)z - 1 -\rho_\star^2 + \alpha L ((1+\delta)-(1-\delta)\rho_\star)).
	\]
	If $\alpha \geq 1/L$, then this is a concave function of $t$, so $F(t) \geq 0$ for all $t \in [-\rho_\star, \rho_\star]$ if and only if $F(-\rho_\star) \geq 0$. But
	\[
	F(-\rho_\star) = 4\lambda_\star \rho_\star(1+\rho_\star)(\rho_\star - (\alpha L(1+\delta)-1)) \geq 0,
	\]
	since $\rho_\star \geq \alpha L(1+\delta)-1$ by assumption. 
	
	On the other hand, if $\alpha < 1/L$, then $F(t)$ is a convex quadratic, so $F(t) \geq 0$ for all $t \in [-\rho_\star, \rho_\star]$ if and only if $F'(\rho_\star) \leq 0$. Now
	\[
	F'(\rho_\star) = -2\alpha \lambda_\star(1-\delta)(2L - \alpha m((3-\delta)L  - (1-\delta)m)),
	\]
	and this is certainly negative for $\alpha < 1/L$.
	
	\item Next, suppose that $\rho_\star := \rho_\GD(\delta) = \alpha L(1+\delta) -1$. In that case, compute
	\[
	\Phi(-\rho_\star, \lambda) =  -L^2 \alpha^2(1+\delta)^2(\lambda \delta - (L-m)\alpha)^2.
	\]
	Nonnegativity of this quantity uniquely determines $\lambda$ by the same formula \eqref{eq:lambdastarconvexGD}. We can again express $\Phi(z,\lambda_\star) = F(\Re z)$, where this time
	\[
	F(t) = 2\lambda_\star (z+\rho_\star)(2(1-\alpha L)z - 2(1-\alpha m) + \alpha (L-m) (2\delta + \alpha L(1-\delta^2))).
	\]
	This is a concave function, so as above we only need to verify that $F(\rho_\star) \geq 0$.  Some algebra gives
	\[
	F(\rho_\star) = 4\lambda_\star \rho_\star(1-\rho_\star)(\rho_\star - (1-\alpha m(1-\delta))).
	\]
	Since $\rho_\star \geq 1-\alpha m(1-\delta)$ by assumption, this quantity is nonnegative provided $\rho_\star \leq 1$.
\end{inparaenum}

In the context of establishing $\rho_\GD(\delta)$-hyperstability via Theorem \ref{theo:popov}, we have shown that \eqref{eq:FDI} holds under the hypotheses of Proposition \ref{prop:GDM}. To verify the $\rho_\GD(\delta)$-minimal stability hypothesis, we combine Lemma \ref{lemm:perturbedminimallystable} with the following result.

\begin{lemm} \label{lemm:underlying}
	Let $A = C = I_n$ and $B = -\alpha I_n$. If $\rho \geq \rho_\GD$ and $\tau \in [0,\rho]$, then there exists $\varepsilon \in [0,1]$ such that the matrix \eqref{eq:minimalaugmented} is $\rho$-Schur.
\end{lemm}
\begin{proof}
	Because of its block structure, we can assume \eqref{eq:minimalaugmented} is the $2\times 2$ matrix
	\[
	\begin{bmatrix}
	1- \alpha((1-\varepsilon) m+ \varepsilon L) & \alpha \varepsilon \tau^2 \\
	(1-\varepsilon)(L-m) & \varepsilon \tau^2 
	\end{bmatrix}.
	\]
	Up to a scalar factor of $\rho^2$, the characteristic polynomial of this matrix as a function of $\rho z$ is
	\[
	z^2 - \rho^{-1}(1-\alpha m - \alpha\varepsilon(L-m) + \varepsilon \tau^2)z +   \varepsilon \tau^2 \rho^{-2}(1-\alpha L).
	\]
	It suffices to show that the roots of this function lie in the open unit disk, or equivalently,
	\begin{gather}
	\rho^{-1}|1-\alpha m +\varepsilon \tau^2 - \alpha \varepsilon (L-m) | < 1 + \varepsilon \tau^2 \rho^{-2}(1-\alpha L), \label{eq:minimal1} \\
	\varepsilon \tau^2 \rho^{-2} |1-\alpha L| <  1. \label{eq:minimal2}
	\end{gather}
	Clearly \eqref{eq:minimal2} is satisfied for $\varepsilon \geq 0$ sufficiently small. If $\rho > |1-\alpha m|$, then \eqref{eq:minimal1} is satisfied with $\varepsilon = 0$. Thus we can assume $\rho = |1-\alpha m|$, which implies $\rho = 1-\alpha m$ since $\rho \geq \rho_\GD$. Now we must verify
	\[
	-1 - \varepsilon \tau^2 \rho^{-2}(1-\alpha L) < \rho^{-1}(1-\alpha m+ \varepsilon \tau^2 -\alpha \varepsilon (L-m) ) < 1+ \varepsilon \tau^2 \rho^{-2}(1-\alpha L).
	\]
	Clearly the first inequality holds for sufficiently small $\varepsilon \geq 0$. As for the second inequality, it suffices to show $\rho(\tau^2 -\alpha(L-m)) < \tau^2 (1-\alpha L)$. But
	\[
	\tau^2(1-\alpha L) - \rho( \tau^2-\alpha (L-m) ) = \alpha (L-m) (\rho - \tau^2),
	\]
	which is strictly positive since $\rho = 1-\alpha m  \in (0,1)$ and $\tau \leq \rho$. Thus it suffices to choose $\varepsilon > 0$ sufficiently small.
\end{proof}

\begin{proof} [Proof of Proposition \ref{prop:GDM}]
	Let $\tau = \rho_\GD(\delta)$ and define $\lambda_\star$ by \eqref{eq:lambdastarconvexGD}.
	In view of the observations above, to show that $(\hat A,B_{ue}; \sigma_\tau + \lambda_\star \sigma_\Delta)$ is $\rho_\GD(\delta)$-hyperstable using Theorem \ref{theo:popov}, we need only verify that $\det \Pi_{ue}(\zeta,z) \not \equiv 0$. The bottom right entry of $\sigma_\tau + \lambda_\star \sigma_\Delta$ in the block decomposition \eqref{eq:QSR} is
	\[
R =	\begin{bmatrix}
		\lambda_\star \delta^2 - 2 & 0 \\
		0& -\lambda_\star 
	\end{bmatrix} \otimes I_n.
	\]
	Since we are assuming that $\alpha \leq 2/((1+\delta)L)$ and $\lambda_\star$ is given by  $\eqref{eq:lambdastarconvexGD}$, this matrix certainly has nonzero determinant, and hence $\det \Pi_{ue}(\infty, \infty) \neq 0$.
\end{proof} 

\subsection{Proof of Lemma \ref{lemm:optimal}}
The proof is essentially contained in \cite{taylor2017smooth}, so we only provide a brief sketch. As in the statement of Lemma \ref{lemm:optimal}, suppose that $0 \leq \delta < 2\sqrt{\kappa}/(\kappa + 1)$ and $\alpha \in (0, \bar \alpha_-)$. Denote by $\rho_\star$ the right-hand side of \eqref{eq:rhoupperbound}. The combination of Proposition \ref{prop:noisyS} and Proposition \ref{prop:weakGD} shows that
\[
\rho_\star^2 = \inf \{\rho^2:\,  \rho^2, \lambda_1, \lambda_2 \geq 0, \ L_1 - \rho^2 L_0 + \lambda_1 M_1 + \lambda_2 M_2 \preceq 0 \},
\]
where we define
\begin{gather*}
L_0 = \begin{bmatrix} 1 & 0 & 0 \\
	0 & 0 & 0 \\
	0 & 0 & 0 \end{bmatrix}, \quad L_1 = \begin{bmatrix} 1 &  -\alpha & -\alpha  \\
-\alpha & \alpha^2 & \alpha^2 \\
-\alpha & \alpha^2 & \alpha^2 \end{bmatrix} , \\ M_1 =   \begin{bmatrix} -2 L m & L+m & 0 \\
L+m & -2 & 0 \\
0 & 0 & 0  \end{bmatrix} , \quad M_2 = \begin{bmatrix} 0 & 0 & 0 \\
0 & \delta^2 & 0 \\
0 & 0 & -1 \end{bmatrix}.
\end{gather*}
This minimization problem is the dual of the primal semidefinite program
\[
\sup \{\mathrm{tr}(L_1 G): \, G \succeq 0, \ \mathrm{tr}(L_0 G) \leq 1, \ \mathrm{tr}(M_1 G) \geq 0, \ \mathrm{tr}(M_2 G) \geq 0 \}.
\]
It is easy to construct a strictly feasible point for the primal problem: first take any nonzero $x_\feas \in \RR^3$ with $\|x_\feas \| < 1$ and let $u'_\feas = (L+m)x_\feas/2$. Then choose $u_\feas = u_\feas' + u_\feas''$, where $u_\feas'' \in \RR^3$ is nonzero and orthogonal to $x_\feas$ (and hence also $u_\feas'$). Finally, choose a nonzero $e_\feas \in \RR^3$ orthogonal to $(x_\feas, u_\feas)$. The matrix
\[
G_\feas := \begin{bmatrix}
	x_\feas & u_\feas & e_\feas
\end{bmatrix}^\top \begin{bmatrix} x_\feas & u_\feas & e_\feas \end{bmatrix}
\]
is positive definite, and 
satisfies $\mathrm{tr}(L_0 G_\feas) = \|x_\feas\|^2 < 1$. Furthermore,
\[
\mathrm{tr}(M_1 G_\feas) = (L-m)^2\|x_\feas\|^2/2 -2 \| u_\feas'' \|^2, \quad \mathrm{tr}(M_2 G_\feas) = \delta^2 \| u_\feas \|^2 - \|e_\feas\|^2.
\]
If we choose $\| u_\feas'' \|$ and $\|e_\feas\|$ sufficiently small, then both terms above are strictly positive, so $G_\feas$ is strictly feasible. Since Slater's condition is satisfied there is no duality gap and the supremum in the primal problem is also equal to $\rho^2_\star$. 

Now observe that if the ambient dimension is $n \geq 3$, then any $\RR^{3 \times 3} \ni G \succeq 0$ can certainly be written in the form
\[
G = \begin{bmatrix}
	x & u & e
\end{bmatrix}^\top \begin{bmatrix} x & u & e \end{bmatrix}
\]
for some $x, u, e \in \RR^{n}$. In particular, if $\rho < \rho_\star$, then there exists $x_0,u_0,e_0 \in \RR^n$ satisfying
\[
\| x_0 - \alpha(u_0 + e_0) \|^2 > \rho^2 \| x_0 \|^2, \quad  \sigma_0(x_0,u_0) \geq 0, \quad \delta^2 \|u_0\|^2 - \|e_0\|^2 \geq 0.
\]
Now given $(x_0, u_0)$ with $\sigma_0(x_0,u_0) \geq 0$, we can always find $f \in \mathcal{F}_0(m,L) \subset \mathcal{S}_0(m,L)$ such that $u_0 = \nabla f(x_0)$. Indeed, according to \cite[Theorem 4]{taylor2017smooth}, there exists $f \in \mathcal{F}_0(m, L)$ such that
\begin{gather*}
f(x_0) = \tfrac{1}{2} \langle u_0, x_0 \rangle \geq 0, \quad f(0) = 0, \\
 \nabla f(x_0) = u_0, \quad \nabla f(0) = 0. 
\end{gather*}
In conclusion, if $\rho < \rho_\star$, then we can find $f \in \mathcal{F}_0(m,L)$ and $x_0, e_0 \in \RR^n$ satisfying
\[
\| x_0 - \alpha(\nabla f(x_0) + e_0) \| > \rho \| x_0 \|, \quad \|e_0\| \leq \delta \| \nabla f(x_0) \|
\]
as desired. \qed

\section{Inexact Triple Momentum Method} \label{sect:TMM}
\subsection{TMM} \label{subsect:exactTMM}

In this section we discuss basic properties of TMM in the exact setting. Recall from \S \ref{subsect:TMMresults} that TMM is defined on $\RR^n$ by the recursion \eqref{eq:degree2}, where $(\alpha,\beta,\gamma)$ are defined by \eqref{eq:TMparameters}. The corresponding convergence rate that holds over $\mathcal{F}(m,L)$ is 
\[
\rho_\TM = 1-1/\sqrt{\kappa}.
\]
By translations we can restrict our attention to $f \in \mathcal{F}_0(m,L)$.

Similar to \S \ref{subsect:nesterov}, TMM admits a state-representation on $\RR^{2n}$ with system matrices
\begin{equation} \label{eq:TMMstatespace}
A = \begin{bmatrix}
	1 + \beta & -\beta \\
	1 & 0
\end{bmatrix} \otimes I_n , \quad B = \begin{bmatrix}
	-\alpha \\ 0  
\end{bmatrix} \otimes I_n, \quad C = \begin{bmatrix}
	1+\gamma & -\gamma
\end{bmatrix} \otimes I_n.
\end{equation}
In \cite{van2017fastest}, the parameters \eqref{eq:TMparameters} were chosen in the frequency domain to satisfy the Jury--Lee criterion
\eqref{eq:jurylee} via pole-zero assignment as follows. Let $H(z) = C(zI-A)^{-1}B$. Since $z \mapsto \det (I - mH(z)) \not \equiv 0$, conjugating by $(I-mH(z))^{-1}$ shows that \eqref{eq:jurylee} is equivalent to
\[
\Re \left [ (1- \tau^2/z) (I - LH(z))(I-mH(z))^{-1} \right]  \succeq 0.
\]
In fact, because of the block structure, for the purposes of verifying \eqref{eq:jurylee} for TMM we can assume that $H(z)$ is scalar, given by
\begin{equation} \label{eq:TMMtransfer}
H(z) = -\alpha \frac{(1+\gamma)z-\gamma}{(z-1)(z-\beta)}.
\end{equation}
We can reduce $(1- \tau^2/z) (1 - LH(z))(1-mH(z))^{-1}$ to a linear fractional transformation of the disk of radius $\rho$ (for a certain $\rho > 0$) onto the left half-plane by an appropriate choice of pole-zero cancellations: there is precisely one way to choose the four unknowns $(\alpha,\beta,\gamma,\rho)$ so that $1-LH(z)$ has roots  $z\in \{-\rho, 0\}$, and $1-mH(z)$ has roots $z\in \{ \rho, \tau^2\}$.
These parameters are
\[
\begin{gathered}
	\alpha = \frac{2(1-\tau^2)}{L(1-\tau^2) + m}, \quad \beta = \frac{(\kappa(1-\tau^2)-1)\kappa \tau^2}{(\kappa(1-\tau^2)+1)(\kappa -1)},  \quad \gamma = \frac{(\kappa(1-\tau^2) + 1)\tau^2}{2(\kappa-1)(1-\tau^2)},  \\ \rho = \frac{ \kappa(1-\tau^2) -1}{\kappa(1-\tau^2) + 1}.
\end{gathered}
\]
Note that we must have $\kappa > 1/(1-\tau^2)$ in order to have $\rho > 0$. Furthermore, $\rho$ is related to $\tau$ by $0 \leq \tau \leq \rho$. With these choices,
\[
(1- \tau^2 /z) (1- LH(z) ) (1-mH(z))^{-1} = \frac{z+\rho}{z-\rho}.
\]
 Since $\rho$ is a decreasing function of $\tau$, the best rate is obtained by setting $\tau = \rho$. This provides an implicit equation for $\rho$ whose solution is $\rho_\TM = 1-1/\sqrt{\kappa}$, and also yields the parameter choices \eqref{eq:TMparameters}.
 
 In \cite{van2017fastest} the authors resort to constructing an explicit Lyapunov function to show that TMM converges at rate $\rho_\TM$. This is because the properties of TMM in the frequency domain are (by design) quite degenerate. In contrast, a refined result like Theorem \ref{theo:popov} enables an analysis entirely within the frequency domain. We apply the Jury--Lee criterion. Since the frequency condition \eqref{eq:jurylee} is satisfied by construction, it remains to show that \eqref{eq:minimalaugmented} is $\rho_\TM$-Schur for some $\varepsilon \in [0,1]$. Because of the block structure, we can assume that the matrix \eqref{eq:minimalaugmented} is given by
\[
\begin{bmatrix}
	\rho_\TM(1+\rho_\TM - (2+\rho_\TM)\varepsilon) & -(1-\varepsilon)\rho_\TM^3 &  \varepsilon m^{-1}\rho_\TM^2 (1-\rho_\TM)^2 (1+ \rho_\TM) \\
	1 & 0 & 0 \\
	\dfrac{ m(1-\varepsilon) \rho_\TM (2+\rho_\TM)}{(1-\rho_\TM)^2 (1+\rho_\TM)} & \dfrac{ -m(1-\varepsilon) \rho_\TM^3 }{(1-\rho_\TM)^2 (1+\rho_\TM)} & \varepsilon\rho_\TM^2
\end{bmatrix}.
\]
The eigenvalues of this matrix are $0, \, \rho_\TM^2$, and $(1-2\varepsilon)\rho_\TM$.
Clearly if $\varepsilon \in (0,1)$, then this matrix is $\rho_\TM$-Schur. Lemma \ref{lemm:jurylee} now applies.

\subsection{Inexact TMM} In order to analyze inexact TMM, we follow \S \ref{subsect:noise} and consider the system
\begin{align*}
	\begin{bmatrix}
		x_{t+1} \\
		v_{t+1}
	\end{bmatrix}	= \hat A \begin{bmatrix}
		x_t \\ v_t
	\end{bmatrix} + B_{ue} \begin{bmatrix}
		u_t \\ e_t
	\end{bmatrix}, \quad \hat A = \begin{bmatrix}
		A & 0_n \\
		LC & 0_n
	\end{bmatrix}, \quad B_{ue} = \begin{bmatrix}
		B & B  \\
		-I_n & 0_n
	\end{bmatrix},
\end{align*}
where $x_t = (\xi_t, \xi_{t-1})$ and $(A,B,C)$ are specified by \eqref{eq:TMMstatespace}. We begin by examining the perturbed Jury--Lee frequency condition \eqref{eq:perturbedFDIconvex}. Fixing $\tau = \rho$ and using the notation of  \S \ref{subsect:noise}, we can write 
\[
2 \Re \left[ (1-\rho^2/z)(1- m H(z))^* (1-LH(z)) \right] = q(\Re z,\rho)
\]
whenever $|z| = \rho$, where
\begin{align*}
q(t,\rho) = -2 (\rho^2 -\rho_\TM^2)(2 \rho_\TM t^2 + (\rho^2 + \rho_\TM^2)t - (1+\rho_\TM)^2 \rho^2).
\end{align*}
We now introduce some additional notation: given $\mu \in \RR$ and a fixed $\rho > 0$, define the linear (in $t$) function
\[
\ell(t, \rho, \mu) \coloneqq  \rho^2 + \mu^2 - 2 \mu t.
\]
If $|z| = \rho$, then we can express
\[
|z - \mu|^2 = \ell(\Re z,\rho, \mu).
\]
Denote by $\Phi(z,\rho,\lambda)$ the left-hand side of \eqref{eq:perturbedFDIconvex}, where we take $\tau = \rho$. Then for $|z| = \rho$, we can write $\Phi(z,\rho,\lambda) = F(\Re z, \rho, \lambda)$, where
\begin{multline*} 
F(t,\rho,\lambda) = \lambda q(t, \rho)  - \lambda^2 \delta^2 \cdot \ell(\Re z, 1) \cdot \ell(\Re z,\beta)
  \\ - \alpha^2 (1+\gamma)^2\left ( (L-m)^2 \ell(\Re z,1) + 2\lambda \delta^2 Lm(1-\Re z) \right) \ell(\Re z, \gamma/(1+\gamma)).
\end{multline*} 
It is easy to see that $F(t,\rho,\lambda)$ is a concave quadratic function of $t$.
Thus, for a given $\rho >0$ and $\lambda \geq 0$, we have $F(t, \rho,\lambda) \geq 0$ for $t \in [-\rho, \rho]$ if and only if $F(\pm \rho, \rho, \lambda) \geq 0$.

\pgfplotsset{scaled x ticks=false,max space between ticks=50}
\begin{figure}
	\begin{center}
		\begin{tikzpicture}[
			declare function ={ 
				theta(\r,\d) = (2-\r)*\r+(2+\r)*\d;
				rho(\r,\d) =(theta(\r,\d)+sqrt(theta(\r,\d)^2+4*\d*\r^2*(2-\r)))/(2*(2-\r));
			},
			]
			\begin{axis}[
				height=6cm, width=7.5cm, title = {$\kappa = 2$}, ylabel=${\rho}$, legend pos = south east,
				ticklabel style = {font=\tiny},
				name=a,
				axis line style={-},
				xmin = 0,
				xmax = 0.507469,
				ymax = 1,
				xlabel style={yshift=0cm},
				ylabel style={rotate=-90}, 
				]
				\addplot[color=black, opacity=.2, line width=5pt] table[col sep=comma,header=false,x index=0,y index=1] {tmm_inexact_2.csv};
				\addplot[dashed, color=black, domain = 0 : 0.507469, samples=50] {rho(0.292893,x)};
				\legend{LMI, $\rho_\TM(\delta)$}
			\end{axis}
			
			\begin{axis}[
				height=6cm, width=7.5cm, title= {$\kappa = 10$}, legend pos = south east,
				ticklabel style = {font=\tiny},
				name=b,
				anchor=north west, at = (a.north east),
				xshift = 1.5cm,
				axis line style={-},
				xmin = 0,
				xmax = 0.132081,
				ymax = 1,
				xlabel style={yshift=0cm},
				ylabel style={rotate=-90},
				]
				\addplot[color=black, opacity=.2, line width=5pt] table[col sep=comma,header=false,x index=0,y index=1] {tmm_inexact_10.csv};
				\addplot[dashed, color=black, domain = 0 : 0.132081, samples=50] {rho(0.683772,x)};
				\legend{LMI,$\rho_\TM(\delta)$}
			\end{axis}
			
			\begin{axis}[
				height=6cm, width=7.5cm, title = {$\kappa = 10^2$}, xlabel=$\delta$, ylabel=${\rho}$, legend pos = south east,
				ticklabel style = {font=\tiny},
				name=c,
				anchor=north west, at = (a.south west),
				yshift = -1.5cm,
				axis line style={-},
				xmin = 0,
				xmax = 0.0296496,
				ymax = 1,
				xlabel style={yshift=0cm},
				ylabel style={rotate=-90},
				]
				\addplot[color=black, opacity=.2, line width=5pt] table[col sep=comma,header=false,x index=0,y index=1] {tmm_inexact_100.csv};
				\addplot[dashed,color=black, domain = 0 : 0.0296496, samples=50] {rho(0.9,x)};
				\legend{LMI, $\rho_\TM(\delta)$}
			\end{axis}
			
			\begin{axis}[
				height=6cm, width=7.5cm, title = {$\kappa = 10^3$}, xlabel=$\delta$, legend pos = south east,
				ticklabel style = {font=\tiny},
				name = d,
				anchor=north west, at = (b.south west),
				yshift = -1.75cm,
				axis line style={-},
				xmin = 0,
				xmax =0.00835168,
				ymax = 1,
				xlabel style={yshift=0cm},
				ylabel style={rotate=-90},
				]
				\addplot[color=black, opacity=.2, line width=5pt] table[col sep=comma,header=false,x index=0,y index=1] {tmm_inexact_1000.csv};
				\addplot[dashed, color=black, domain = 0 : 0.00835168, samples=50] {rho(0.968377,x)};
				\legend{LMI, $\rho_\TM(\delta)$}
			\end{axis}
		\end{tikzpicture}
	\end{center}
	\caption{Plots of $\rho_\TM(\delta)$ and the numerically computed rate obtained by solving \eqref{eq:LMI} for $P \succeq 0$ using the off-by-one IQC (over all admissible $\tau$). This supports the conjecture that $\rho_\TM(\delta)$ is actually optimal with respect to the inexact Jury--Lee criterion. \label{fig:tmm_inexact}}
\end{figure}
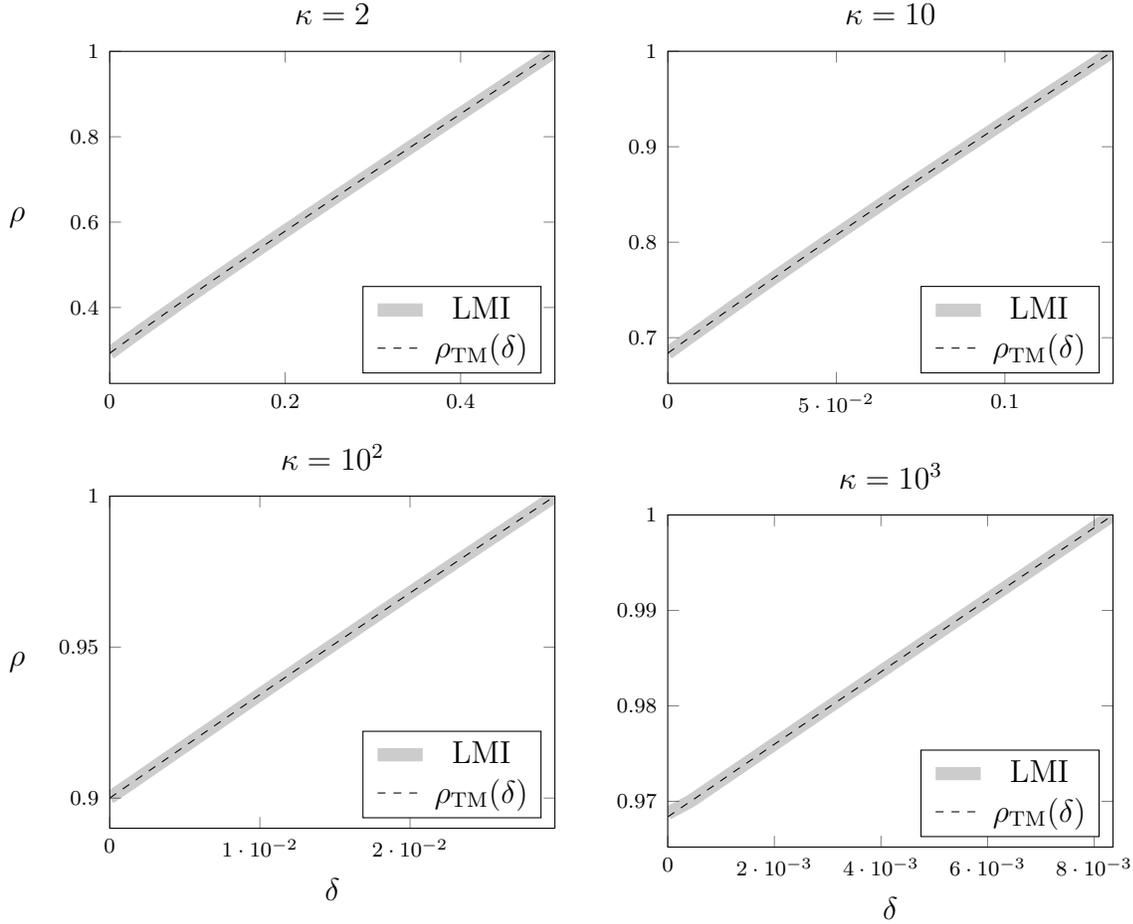

Next, define
\[
\rho_\star \coloneqq \inf \{ \rho \geq \rho_\TM: F(\pm \rho, \rho, \lambda) \geq 0 \text{ for some } \lambda \in (0, 2\delta^{-2}] \}.  
\]
By arguing as in \S \ref{subsect:GDSml}, if the optimal value $\rho_\star$ is finite, then $F(t,\rho_\star, \lambda_\star) \geq 0$ for all $t \in [-\rho_\star, \rho_\star]$ and some $\lambda_\star \in (0,2\delta^{-2}]$. Since $\rho_\TM \in (0,1)$, it is easy to check that
\[
F(0,\rho,2\delta^{-2}) <0
\]
 for any $\rho > 0$. Consequently, $\lambda_\star \in (0,2\delta^{-2})$. 
 
 As in \S \ref{subsect:GDSml}, we must have that at least one of $F(\pm \rho_\star, \rho_\star, \lambda_\star) = 0$. Furthermore, suppose that $F(-\rho_\star, \rho_\star, \lambda_\star) = 0$ but $F(\rho_\star, \rho_\star, \lambda_\star) > 0$. Since $\lambda_\star \in (0, 2/\delta^2)$, we must have $\partial_\lambda F(-\rho_\star, \rho_\star, \lambda_\star) = 0$, otherwise by perturbing $\lambda$ slightly we could contradict the definition of $\rho_\star$. The same argument applies if we replace $t = -\rho_\star$ with $t = \rho_\star$. Thus there are three (not mutually exclusive) possibilities:
 \begin{enumerate} \itemsep6pt
 	\item $F(-\rho_\star, \rho_\star, \lambda_\star) = 0 = F(\rho_\star, \rho_\star, \lambda_\star)$, 
 		\item $F(\rho_\star, \rho_\star, \lambda_\star) = 0 = \partial_\lambda F(\rho_\star, \rho_\star, \lambda_\star)$,
 	\item $F(-\rho_\star, \rho_\star, \lambda_\star) = 0 = \partial_\lambda F(-\rho_\star, \rho_\star, \lambda_\star)$.
  \end{enumerate} 
We have been unable to analytically characterize which of these conditions is optimal depending on the values of the parameters. However, some experimentation suggests that the third condition always leads to the smallest value of feasible $\rho$. Thus, as a relaxation, we find the smallest $\rho$ satisfying 
\[
F(-\rho, \rho, \lambda) = 0 = \partial_\lambda F(-\rho, \rho, \lambda)
\]
for some $\lambda \in (0,2\delta^{-2})$, and then show that \eqref{eq:FDI} indeed holds.

Due to the quadratic nature of $F(t,\rho, \lambda)$ in $\lambda$, it suffices to solve for $\rho$ satisfying 
\[
\disc(F(\rho,\rho,\lambda); \lambda) = 0.
\]
Define the quadratic functions
\begin{align*}
a_{\pm}(\rho) & \coloneqq  -(2-\rho_\TM )\rho^2 + ( \rho_\TM (2-\rho_\TM) \pm \delta(2+\rho_\TM))\rho \pm \delta \rho_\TM^2, \\
b_{\pm}(\rho) & \coloneqq -(2-\rho_\TM)\rho^2 - (\rho_\TM (2+\rho_\TM -\rho_\TM^2) \mp \delta (2-3\rho_\TM + \rho_\TM^3)) \rho  \\
 &- \rho_\TM^2 ((2-\rho_\TM)\rho_\TM \mp \delta(1-\rho_\TM)^2) .
\end{align*}
Then the discriminant admits the factorization
\[
\mathrm{disc}(F(-\rho,\rho, \lambda); \lambda) = 4\rho^{-2} (2-\rho_\TM)^{-4} (\rho + \rho_\TM^2)^2 a_+(\rho) a_{-}(\rho) b_{+}(\rho)b_{-}(\rho).
\]
Only $a_{+}$ and $b_{+}$ can have positive roots for $\delta  > 0$. Furthermore, while $b_{+}$ can have positive roots, the corresponding unique value of $\lambda$ for which $F(-\rho,\rho,\lambda) = 0$ is actually negative for $\rho > 0$. Thus the only feasible value of $\rho$ is the unique positive root $a_{+}$  (such a root always exists, since $a_{+}$ is a concave quadratic and $a_{+}(0) > 0$). In the notation of Proposition \ref{prop:TMM}, this positive root is
\[
\rho_\TM(\delta)  \coloneqq \frac{\theta_\TM + (\theta_\TM^2 + 4\delta \rho_\TM^2(2-\rho_\TM))^{1/2}}{2(2-\rho_\TM)},
\]
where $\theta_\TM = (2-\rho_\TM)\rho_\TM + (2+\rho_\TM)\delta$.

Next, we verify that $\rho_\TM(\delta)$ is actually feasible for \eqref{eq:FDI}. To ease the notational burden, let us write $\rho_0 = \rho_\TM(\delta)$. Let $\lambda _0$ be the unique value satisfying
\[
F(-\rho_0, \rho_0, \lambda_0) = 0.
\]
 It suffices to show that $\partial_t F(0, \rho_0, \lambda_0) \geq 0$. 
Perhaps the simplest way to show this is to first express $\delta$ and $\lambda_0$ as functions of $\rho_0$:
\begin{equation} \label{eq:TMMlambdadelta}
\begin{gathered}
\delta = \frac{\rho_0 \rho_\TM (\rho_0 - \rho_\TM)(2- \rho_\TM)^2}{ \rho_0 (2+\rho_\TM) + \rho_\TM^2}, \\ \lambda_0 = \frac{\rho_\TM(\rho_\TM^2 + \rho_0(2+\rho_\TM))^2}{\rho_0(\rho_0^2(2-\rho_\TM)-2\rho_0 \rho_\TM (1-\rho_\TM) + \rho_\TM^3)}.
\end{gathered}
\end{equation} 
Now plug \eqref{eq:TMMlambdadelta} back into the derivative. It is then easy show that all the Taylor coefficients of  the polynomial $\rho \mapsto \partial_t F(0,\rho,\lambda_0)$ at $\rho = \rho_\TM$ are nonnegative, which implies that the derivative is nonnegative for any value of $\rho_0 \geq \rho_\TM$. For the sake of brevity we omit the details.

\begin{proof} [Proof of Proposition \ref{prop:TMM}]
To finish the proof of Proposition \ref{prop:TMM}, we must show that $(\hat A, B_{ue}; \sigma_{\rho_0} + \lambda_0 \sigma_\Delta)$ is actually $\rho_0$-hyperstable. The minimal stability hypothesis of Theorem \ref{theo:popov} follows from Lemma \ref{lemm:perturbedminimallystable} and the results of \S \ref{subsect:exactTMM}. It is also easy to see that $\lambda_0 \in (0,2\delta^{-2})$ for any $\rho_\TM, \rho_0  \in (0,1)$, so we can apply Theorem \ref{theo:popov} as in the proof of Proposition \ref{prop:GDM} (which can be found at the end of \S \ref{subsect:inexactFmL}).
\end{proof}

	\bibliographystyle{alphanum}
	
	\bibliography{frequency_criteria}

\end{document}